\documentclass[11pt,reqno]{amsart}
\usepackage[a4paper]{geometry}
\usepackage{amsthm,hyperref,lmodern,mathrsfs}
\usepackage[latin1]{inputenc}
\usepackage[T1]{fontenc}
\usepackage{graphicx,caption, subcaption}
\usepackage[frenchb,english]{babel}
\usepackage{calc}
\usepackage{amsmath,amsthm,mathabx} 
\usepackage{tikz}
\usepackage{enumerate}
\usepackage{cancel}

\usepackage{color}
\usepackage{url}
\usepackage{times, amssymb, amscd, mathrsfs, graphicx, color,stmaryrd} 
\usepackage{enumerate}

\usepackage{enumitem}

\newtheorem{theo}{Theorem}[section]
\newtheorem{prop}[theo]{Proposition}

\newtheorem{coro}[theo]{Corollary}
\newtheorem{lem}[theo]{Lemma}
\newtheorem{exe}[theo]{Example}

\newtheorem{hypo}[theo]{Hypothesis}
\newtheorem{Rq}[theo]{Remark}

\newcommand{\tcrg}{\textcolor{black}}

\newcommand{\tcr}{\textcolor{black}}
\newcommand{\1}{\mathbf{1}}
\newcommand{\N}{\mathbb{N}}                                              
\newcommand{\R}{\mathbb{R}}
\newcommand{\ER}{\mathbb{R}}
\newcommand{\cim}{\partial}
\newcommand{\cal}{\mathcal}                                

\newcommand{\PE}{\mathbb{P}}
\newcommand{\E}{\mathbb{E}}
\newcommand{\ES}{\mathbb{E}}
\newcommand{\XX}{\xi}
\newcommand{\MM}{\mathcal{E}}
\newcommand{\Aa}{{\cal A}}
\newcommand{\beq}{\begin{equation}}
\newcommand{\eeq}{\end{equation}}

\newcommand{\nocontentsline}[3]{}
\newcommand{\tocless}[2]{\bgroup\let\addcontentsline=\nocontentsline#1{#2}\egroup}

\definecolor{darkred}{rgb}{0.9,0.1,0.1}
\def\comment#1{\marginpar{\raggedright\tiny{\textcolor{darkred}{#1}}}}
\title[Approximation of Quasi-Stationary Distributions]{Stochastic approximation of Quasi-Stationary Distributions on compact spaces  and Applications}

\author{Michel Benaim, Bertrand Cloez, Fabien Panloup }
\address{M. Benaim: Institut de mathématiques
Rue Emile-Argand 11,
2007 Neuchâtel, CH.}
\address{B. Cloez: MISTEA, INRA, Montpellier SupAgro, Univ. Montpellier, Montpellier,
France}
\address{F. Panloup: LAREMA-UMR CNRS 6093, Université d'Angers, 2, Bd Lavoisier, 49045 ANGERS Cedex 01.} %
\email{\url{michel.benaim(at)unine.ch}, \url{bertrand.cloez(at)supagro.fr}, \url{fabien.panloup(at)univ-angers.fr} }%

\date{ Compiled \today}

\begin{document}
\maketitle

\begin{abstract}
{\tcr{As a continuation of} a recent paper, dealing with finite Markov chains,
this paper proposes and analyzes a recursive algorithm for the approximation
 of the quasi-stationary distribution of a general Markov chain living  on a compact metric space killed in finite time. The idea is to run the process until extinction and then to bring it back to life at a position randomly chosen according to the  (possibly weighted) empirical occupation measure of its past positions.  General conditions are given ensuring the convergence of this measure to the quasi-stationary distribution of the chain.
  We then apply this method  to the numerical approximation of the quasi-stationary distribution of
   a diffusion process killed on the boundary of a compact set. }
   {Finally, the sharpness of the assumptions is illustrated through the study of the algorithm in a non-irreducible setting}.

\end{abstract}
{\footnotesize
\noindent\textbf{Keywords.}  Quasi-stationary distributions ; stochastic approximation ; reinforced random walks ;
random perturbations of dynamical systems; extinction rate; Euler scheme.

\noindent\textbf{AMS-MSC.}  65C20; 60B12; 60J10, Secondary 34F05; 60J20; 60J60.

}


{\footnotesize %

\medskip


}
\section{Introduction}
Numerous models, in ecology and elsewhere, describe the temporal evolution of a system  by a Markov  process which eventually gets  killed  in  finite time. In population dynamics, for instance,  extinction in finite time
  is a typical effect of  finite population sizes. However, when  populations are large, extinction usually occurs  over very large time scales and the relevant phenomena are given by the behavior of the process conditionally to its non-extinction.

  More formally, let $(\xi_t)_{t\ge0}$ be a Markov process with values in ${\MM}\cup\{\cim\}$ where \tcr{$\MM$ is a metric space and $\cim \not \in \MM$} denotes an absorbing point (typically, the extinction set or the \tcr{complement} of a domain). Under appropriate assumptions, there exists a  distribution $\nu$ on $\MM$ (possibly depending on the initial distribution of $\xi$) such that
   \begin{equation}
   \label{eq:QLD}
   \lim_{t \rightarrow \infty} \PE( \xi_t \in . | \xi_t \neq \cim) = \nu(\cdot).
   \end{equation}
Such a distribution well describes the behavior of the process before extinction, and is necessarily (see e.g~\cite{MV12}) a  {\em quasi-stationary distribution} (QSD) in the sense that
$$\PE_\nu( \xi_t \in \cdot | \xi_t \neq \cim)= \nu(\cdot).$$
We refer the reader  to the survey paper \cite{MV12} or  the book \cite{CMMM11} for general background and a  comprehensive introduction to  the subject.  
 \smallskip

 The simulation and numerical approximation of  quasi-stationary distributions have received a lot of attention in the recent years and led to the development and  analysis of a class of particle systems algorithms known in the literature as
 {\em Fleming-Viot algorithms}  (see \cite{BHM00,CT13,MM00,V11}).
  The principle of these algorithms is to run a large number of particles independently until one is killed and then to replace the killed particle by an offspring whose location is
   randomly (and uniformly) chosen  among the locations of the other (alive) particles.
In the limit of an infinite number of particles, the (spatial) empirical occupation measure of the particles approaches the law of the process conditioned to never be absorbed\tcrg{; see for instance \cite[Theorem 1]{V11}} . Combined with (\ref{eq:QLD}), this  gives a method for  estimating  the QSD of the process.

 \tcr{In a related context the new paper \cite{2016arXiv160903436P}  demonstrates the importance  to simulate  QSDs in computational statistics as an alternative approach to classical MCMC simulations.}
 
\smallskip

\noindent Recently, in the setting of finite state Markov chains, Benaim and Cloez \cite{BC15}  (see also \cite{BGZ}) analyzed and generalized an alternative approach  introduced in \cite{AFP} in which the {\bf spatial occupation measure} of a set of particles is replaced by the {\bf temporal occupation measure} of a single particle.
Each time the particle is killed
it is risen at a location  randomly chosen according to its temporal occupation measure.
 The details of the construction are recalled in Section \ref{setandmainres}. \smallskip

\noindent {The objective of this paper is twofold:  on one hand,  we aim at extending the results of \cite{BC15}  to the setting of Markov chains with values in a  general space, being killed when leaving a compact domain. \tcrg{Indeed, up to our knowledge, in all the previous works for this algorithm \cite{AFP, BC15,BGZ}, the state space $\MM$ is finite.} On the other hand,  we also explore various applications: we propose and investigate a numerical procedure, based on an Euler discretization, for approximating QSD of diffusions. 

\tcrg{ \noindent {In contrast with the Fleming-Viot particle system, this algorithm requires less calculus but more memory. Also, it only depends on only one parameter (the time) and then approximates in the same time the conditioned dynamics and its long time limit;} \tcr{in particular, it does not require to calibrate simultaneously the number of particles and the time parameter as in the standard Fleming-Viot approach}. Instead, in view of a convergence result for this algorithm, one needs to obtain some properties which are similar to  the commutation of the limits of large particles and of the long time for the Fleming-Viot algorithm. \tcrg{ For the particle system, this type of problem is not completely solved in general but some results have been obtained in some particular settings; see for instance \cite{BHM00,CT16,CT13,FM07, GK04, V11-ejp}.} \tcr{Note that} \tcrg{an example where} \tcr{the  commutation property does not hold} \tcrg{is exhibited in Section \ref{sect:exe}. Besides, let us cite \cite{DG99}, \cite[Section 3]{BC15} or \cite{OV16} which give three different discrete-time Fleming-Viot type algorithms where the double limit is either not proved or proved under restrictive assumptions. Another difference is that the Fleming-Viot process is often developed in continuous-time although our stochastic approximation scheme is in discrete time. As a consequence it is difficult to compare our assumptions on the transition kernel with the ones of the  articles mentioned previously. However, implementing the methods of \cite{BHM00,GK04,V11-ejp} requires a discretization and then leads to the QSD of an Euler-type sequence instead of the one of the target diffusion process. Theorem \ref{prop:diff1} corroborates the consistence of their methods and also shows the consistence of our algorithm. 
   }}

\textbf{Outline.} The paper is organized as follows:
 In Section \ref{setandmainres} we detail the general framework, the hypotheses  and state our main results.
In  Section \ref{sect:appli}, we first discuss our  assumptions in the simple case of finite Markov chains and then focus
 on \tcr{the application to the }  numerical approximation of QSDs for diffusions
 (including theoretical results and numerical tests), 
  The sequel of the paper (Sections \ref{sect:Prelim}, \ref{sect:ODE}, \ref{sect:APT}, and \ref{sec:PMR})
   is mainly devoted to the proofs and the details about their sequencing will be given at the end of Section \ref{sect:appli}.
    We end the paper by some potential extensions of this work to some more general settings such
    as non-compact domains or continuous-time reinforced strategies.

\section{Setting and Main Results}\label{setandmainres}
\subsection{Notation and Setting}
Let ${\MM}$ be a compact metric space\footnote{For comments about a possible extension to the non-compact case, see Section \ref{sec:extensions}} equipped with its Borel $\sigma$-field ${\cal B}(\MM).$ Throughout, we let ${\cal B}(\MM, \R)$ denote the set of real valued bounded measurable functions on $\MM$ and ${\cal C}(\MM, \R) \subset {\cal B}(\MM, \R)$ the subset of continuous functions. For all  $f \in {\cal B}(\MM, \R)$ we let  $\|f\|_{\infty} = \sup_{x \in \MM} |f(x)|$ and we let $\1 $ denote the constant map $x \mapsto 1.$
  We let ${\cal P}(\MM)$ denote the space of (Borel) probabilities over $\MM$ equipped with the topology of weak* convergence.  For all $\mu \in {\cal P}(\MM)$ and $f \in  {\cal B}(\MM, \R),$ or $f$ nonnegative measurable, we write $\mu(f)$ (or $\mu f$) for $\int_{\MM} f d\mu.$  Recall that $\mu_n \rightarrow \mu$ in ${\cal P}(\MM)$ provided $\mu_n(f) \rightarrow\mu(f)$ for all $f \in {\cal C}(\MM, \R),$ and that (by compactness of $\MM$ and Prohorov Theorem), ${\cal P}(\MM)$ is a compact metric space (see e.g~ \cite[Chapter 11]{dudley}). \smallskip

A {\em  sub-Markovian kernel} on $\MM$ is a map $Q : \MM \times {\cal B}(\MM) \mapsto [0,1]$ such that for all $x \in \MM,$ $A \mapsto Q(x, A)$ is a nonzero   measure  (i.e~$Q(x,\MM) > 0$) and for all $A \in {\cal B}(\MM), x \mapsto Q(x,A)$ is measurable.  If furthermore $Q(x, \MM) = 1$ for all $x \in \MM,$ then $Q$ is called a {\em Markov (or Markovian) kernel}.

Let $Q$ be  a sub-Markovian (respectively Markovian) kernel. For every $f \in {\cal B}(\MM,\R)$  and  $\mu \in {\cal P}(\MM),$  we let $Qf$ and $\mu Q$ respectively denote the map and measure defined by
$$Qf(x) = \int_{\MM} f(y) Q(x,dy)\quad \textnormal{and}\quad \mu Q(\cdot) = \int_{\MM} \mu(dx)Q(x,\cdot).$$
If   $Qf \in {\cal C}(\MM,\R)$ whenever $f \in {\cal C}(\MM,\R),$ then $Q$ is said to be
{\em Feller}.
For all  $n \in \N,$ we let $Q^n$ denote  the sub-Markovian (respectively Markovian) kernel  recursively defined by $$Q^{n+1}(x,\cdot) = \int_{\MM} Q(y, \cdot)  Q^n(x,dy) \mbox{ and } Q^0(x,\cdot) = \delta_x.$$
A probability $\mu \in {\cal P}(\MM)$ is called  a {\em quasi-stationary distribution} (QSD) for $Q$ if $\mu$ and $\mu Q$ are proportional or, equivalently, if,
$$\mu = \frac{\mu Q}{\mu Q \1}.$$
The number
\begin{equation}
\label{eq:defTheta}
\Theta(\mu) : = \mu Q \1
\end{equation} is called the {\em extinction rate} of $\mu.$

Note that when $Q$ is Markovian, a quasi stationary distribution is {\em stationary} (or {\em invariant}) in the sense that
$\mu = \mu Q.$ {In this case $\Theta(\mu)=1$, otherwise $\Theta(\mu) <1$.}
\medskip

From now on and throughout the remainder of the paper  we assume  given a Feller sub-Markovian kernel $K$ on $\MM.$

Let $\cim \notin \MM$ be a {\em cemetery point}. Associated to $K$ is the Markov kernel $\widebar{K}$ on $\MM \cup \{ \cim \}$ defined, for all $x \in \MM, A \in {\cal B}(\MM),$ by
\begin{equation}\label{eq:defkbar}
\left\{
  \begin{array}{ll}
    \widebar{K}(x,A) = K(x,A), &   \\
    \widebar{K}(x,\{\cim\}) = 1- K(x, \MM), & \mbox{ and }\\
    \widebar{K}(\cim,\{\cim\}) = 1 . &
  \end{array}
\right.
\end{equation}
The kernel $\widebar{K}$ can be understood as the transition kernel of a Markov chain $(Y_n)_{n\ge0}$  on $\MM \cup \{ \cim \}$
 whose transitions in $\MM$ are given by $K$ and which is "killed" when it leaves ${\MM}.$

Let $\delta :{\MM} \mapsto [0,1]$ be  the function defined by
 $$\delta = \1 - K \1.$$ That is, for every $x\in\MM$,
\begin{equation}
\label{eq:defdelta}
\delta(x)=\widebar{K}(x,\{\cim\}) = 1- K(x, \MM).
\end{equation}
For a given $\mu \in {\cal P}({\MM}),$ we let  $K_\mu$ denote  the Markov kernel on $\MM$ defined by
$$ K_\mu(x,A)=K(x,A)+\delta(x)\mu(A)$$
for all $x\in \MM$ and $A \in {\cal B}(\MM).$
Equivalently,  for every $f \in {\cal B}(\MM,\R),$
$$K_\mu f(x)=Kf(x)+\delta(x)\mu(f).$$
 The chain induced by $K_{\mu}$ behaves like $(Y_n)$ until it is killed and  then is  redistributed in $\MM$ according to $\mu.$ Note that  $K_\mu$ inherits the Feller continuity from $K$.
 For the sequel, an important feature of $K_{\mu}$  is that   $\mu$ is a QSD for $K$ if and only if it is invariant for $K_{\mu}$  (see Lemma \ref{lemme1} for details).

Let $(\Omega, {\cal F}, \PE)$ be a probability space equipped with a filtration $\{{\cal F}_n\}_{n \geq 0}$ (i.e an increasing family of $\sigma$-fields). We now consider an $\MM$-valued random process $(X_n)_{n\ge0}$  defined on $(\Omega, {\cal F}, P)$ adapted to $\{{\cal F}_n\}_{n \geq 0}$ such that
\begin{equation}\label{dynamicsxn}
X_0=x\in{\MM}\quad\textnormal{ and $\forall\,n\ge0$,}\quad \PE(X_{n+1}\in dy|{\cal F}_n)= K_{\mu_n}(X_n,dy),
\end{equation}
where
\begin{equation}\label{dynamicsmun}
\mu_n=\frac{\sum_{k=0}^n \eta_k \delta_{X_{k}}}{ \sum_{k=0}^n\eta_k}
\end{equation}
is a {\em weighted occupation measure}.
Here  $(\eta_n)_{n\geq0}$ is a sequence of {positive} numbers satisfying certain conditions that will be specified below (see Hypothesis \ref{hypogain}).

With the definition of $K_\mu$, this  means that whenever the original process $(Y_n)_{n\ge0}$ is killed, it is  redistributed in ${\MM}$
according to its weighted empirical  occupation measure $\mu_n$.
 Note that such a process is a type of {\em reinforced random walk} (see e.g~ \cite{P07}).
  It is reminiscent of  {\em interacting particle systems} algorithms  used for the simulation of QSDs such as the so-called {\em Fleming-Viot algorithm} (see \cite{BHM00, CT13, V11} and \cite[Section 3]{BC15}). However, while these latter algorithms involve a large number of particles whose individual dynamics depend on the {\em spatial occupation measure} of the particles, here  there is a single particle whose dynamics depends on its {\em temporal occupation measure}.
From a simulation point of view, this is of potential interest, suggesting fewer computations (but more memory) and leading to a recursive method
 which avoids  (at least in name) the trade-off between the number of particles and the time horizon induced by Fleming-Viot algorithm.

Set, for $n\geq 0,$  $$\gamma_n =\frac{\eta_{n}}{\sum_{k=0}^n \eta_k}.$$  The occupation measure can then be computed recursively as follows:
\begin{equation}
 \label{eq:mupasapas}
\mu_{n+1}=(1-\gamma_{n+1})\mu_n+\gamma_{n+1}\delta_{X_{n+1}}.
\end{equation}
Under appropriate irreducibility assumptions (see Hypothesis \ref{hypoK} below), $K_{\mu}$ admits a unique invariant probability $\Pi_{\mu}.$ Owing to the above characterization of QSDs as fixed points of $\mu\mapsto \Pi_\mu$, we choose to rewrite the evolution of $(\mu_n)$ as:
\begin{equation}
\label{eq:algosto}
\mu_{n+1}=\mu_n+\gamma_{n+1}( -\mu_n + \Pi_{\mu_n}) +\gamma_{n+1} \varepsilon_{n}
\end{equation}
where  $\varepsilon_n=\delta_{X_{n+1}}-\Pi_{\mu_{n}}$. The process  $(\mu_n)$ is therefore a  {\em stochastic approximation algorithm} associated to  the ordinary differential equation (ODE) (for which rigorous sense will be given in Section \ref{sect:ODE}):
\begin{equation}
\label{eq:ODEE}
\dot{\mu}=-\mu+\Pi_{\mu}.
\end{equation}
The almost sure convergence of $(\mu_n)$ towards $\mu^\star$ (the QSD of $K$) will then be achieved by  proving that  :
\begin{itemize}
\item[(i)] The asymptotic dynamics of $(\mu_n)_{n\ge0}$ matches with that of solutions of the above ODE: more precisely, $(\mu_n)_{n\ge0}$ is  (at a different scale) an \textit{asymptotic pseudo-trajectory} of the ODE (in the sense of Benaim and Hirsch \cite{BH96}, see \cite{B99} for background).
\item[(ii)] The set ${\rm Fix(\Pi)}=\{\mu\in{\cal P}({\MM}), \mu=\Pi_\mu\}$   reduces to $\mu^\star$ and is a global attractor of the ODE.
\end{itemize}

This strategy was applied in \cite{BC15} when $\MM$ is a  finite set. However,  the proofs in  \cite{BC15} strongly rely on finite dimensional arguments that cannot be applied in this more general setting and the new proofs will require a careful study of the kernel family $(K_\mu)_{\mu}$.\\

\subsection{Main results}
 We first summarize the standing assumptions under which our main results will be proved. We begin by the assumptions on $(\gamma_n)_{n\ge1}$.
 \begin{hypo}[Standing assumption on $(\gamma_n)$]
\label{hypogain}
\noindent
 The sequence $(\gamma_n)_{n\ge0}$ appearing in equation (\ref{eq:mupasapas}) is a non-increasing sequence such that
\begin{equation}
\label{eq:gammahyp}
\sum_{n\geq 0} \gamma_n = + \infty \quad \mbox{and} \quad \lim_{n\rightarrow+\infty}\gamma_n \ln(n)=0.
\end{equation}
\end{hypo}
\noindent {The typical sequence is given by $\gamma_n= \frac{1}{n+1}$}, {which corresponds to $\eta_n=1$ for all $n\ge1$}.\smallskip

 Now, let us focus on the assumptions on the sub-Markovian kernel $K$. We say that  a non-empty set $A \in  \cal{B}(\MM) \cup \{\cim\}$  is {\em accessible} if for all $x \in \MM$
 $$\sum_{n \geq 1} \widebar{K}^n(x,A) > 0.$$
It is called a {\em weak\footnote{this is a mildly weaker definition than the usual definition of small or petite sets (see e.g~\cite{duf00,MT93}) } small set} if $A \subset \MM$ and  there exists a probability measure $\Psi$ on $\MM$ and
 $\epsilon > 0$ such that for all $x \in A$
\begin{equation}
\label{eq:smallset}
\sum_{n \geq 1} K^n(x,dy) \geq \epsilon \Psi(dy).
\end{equation}
\begin{hypo}[Standing assumptions on $K$]
\label{hypoK}
\noindent
\begin{itemize}
\item{$\mathbf{(H_1)}$ } $K$ is Feller.
\item{$\mathbf{(H_2)}$} The cemetery point $\{\cim\}$ is accessible.
\end{itemize}
\end{hypo}
Assumptions $\mathbf{H_1}$ and $\mathbf{H_2}$  imply the existence of  a quasi-stationary distribution but are not sufficient to ensure its uniqueness (see the example developed in Subsection \ref{sect:exe}). For this, we require the supplementary assumptions below
\begin{hypo}[Additional assumptions on $K$]
\label{hypoH3}
\noindent
\begin{itemize}
\item{$\mathbf{(H_3)}$} There exists an open accessible {weak} small set $U$.
\item{$\mathbf{(H_4)}$} There exists a non increasing convex function $C : \R^+ \mapsto \R^+$ satisfying
\begin{equation}\label{eq:condC1}
\int_0^{\infty} C(s) ds = \infty
\end{equation}
such that
$$\frac{\Psi(K^n \1)}{\sup_{x \in \MM} K^n \1(x)} \geq C(n)$$ where
$\Psi$ \tcr{satisfies} equation (\ref{eq:smallset}).
\end{itemize}
\end{hypo}

Roughly, the latter hypothesis  stipulates that the rate at which the process dies is uniformly  controlled, in terms of the initial point. This  is motivated by the recent work of Champagnat and Villemonais \cite{CV14} in which it is proved that under  mildly stronger versions of $\mathbf{H_3}$ (namely, $K^l(x,\dot) \geq \epsilon \Psi$  for some $l$ independent of $x$) and $\mathbf{H_4}$ (namely $C(t) \geq c > 0$)
\tcr{the sequence of conditioned laws defined by} $$\mathbb{P}_x (Y_n \in \cdot | Y_n \in \MM) = \frac{K^n(x, \cdot)}{K^n \1 (x)}, \quad n\ge0,$$
   converges, as $n \rightarrow \infty,$ exponentially fast to a (unique) QSD. Here, Assumption $\mathbf{H_4}$ which does not require the function $s\mapsto C(s)$ to be lower-bounded does certainly not guarantee the exponential rate but is a sharper and almost necessary assumption for the uniqueness and the \textit{attractiveness} of the QSD (on this topic, see also Remark \ref{rem:sufficientcondition} below and Proposition \ref{th:sharpH4}). More precisely,  it will be shown that under  $\mathbf{H_3}$ and  $\mathbf{H_4}$, the semiflow induced by  (\ref{eq:ODEE}) is globally asymptotically stable (i.e ${\rm Fix}(\Pi)$  \tcr{is a singleton} and is a global attractor).

\smallskip

%

\begin{Rq}[{Sufficient condition}]\label{rem:sufficientcondition} {\rm
A simple condition ensuring  Hypothesis \ref{hypoH3}  is that, for some $l \geq 1,$ constants $c_1,c_2>0$ and some $\Psi \in {\cal P}(\MM)$
\begin{equation}
\label{eq:minmaj}
c_1 \Psi (dy) \leq K^\ell(x,dy) \leq c_2 \Psi(dy).
\end{equation}
 Indeed, under (\ref{eq:minmaj}), for  $n\ge\ell,$ $$c_2 \Psi (K^{n-\ell} \1)\ge K^n \1 \geq c_1 \Psi (K^{n-\ell} \1)$$ while for
$n \le \ell,$ $\1 \ge K^n \1 \ge K^\ell \1 \geq c_1.$ Hence $C(t) = \min\left(\dfrac{c_1}{c_2}, c_1 \right) >0.$
Note that \eqref{eq:minmaj}, which is usual in the literature (see $e.g.$ \cite[Theorem 3.2]{berglund_landon}), is satisfied if $K^\ell$ admits a continuous and positive density with respect to a positive reference measure.}\smallskip
\end{Rq}
\noindent Finally,  note that in Hypotheses \ref{hypoK} and \ref{hypoH3} there is no aperiodicity assumption on $K$.\smallskip

We are now able to state our main general result about the convergence of the empirical measure $(\mu_n)_{n\ge0}$ towards the QSD.

\begin{theo}[Convergence of the algorithm]\label{theo1}  Assume Hypotheses \ref{hypogain}, \ref{hypoK} and \ref{hypoH3}. Then, $K$ has a unique QSD  $\mu^\star$ and the sequence $(\mu_n)_{n\ge0}$ \tcr{defined by \eqref{dynamicsmun}} converges $a.s.$ in ${\cal P}({\MM})$ towards $\mu^\star$.
\end{theo}
In fact, the previous setting also leads to the convergence in distribution of the reinforced random walk:

\begin{theo}[Convergence in distribution of $(X_n)_{n\ge0}$]
\label{th:cvmarche}
Suppose that the assumptions of Theorem \ref{theo1} hold. Then, for any starting distribution $\alpha$, $(X_n)_{n\ge0}$ \tcr{defined by  \eqref{dynamicsxn}} converges in distribution to $\mu^\star$.
\end{theo}

The two above results thus show that the algorithm both produces a way  to approximate $\mu^\star$ and also
to sample a random variable with this distribution. The convergence in law of $(X_n)_{n\ge0}$
may appear surprising due to the lack of aperiodicity assumption for $(Y_n)_{n\ge0}$. To overcome this problem,
 we prove in fact that $(X_n)_{n\ge0}$ gets asymptotically this property.


  The extinction rate $\Theta(\mu^\star)${, defined in \eqref{eq:defTheta},}  can be estimated through the same procedure. For this, we need to keep track of the times at which a "resurrection" occurs. We then construct $(X_n)$ as follows. Let  $((U_n, X_n))$ be a process adapted to $\{{\cal F}_n\},$ with $U_n \in \MM \cup \{\cim\}, X_n \in \MM,$ satisfying $X_0 = U_0 = x,$
$$\PE(U_{n+1}\in dy| {\cal F}_n)=\widebar{K}(X_n,dy)$$ and
$$X_{n+1}=V_{n+1}1_{\{U_{n+1}=\cim \} }+U_{n+1}1_{\{U_{n+1}\in \MM\}},$$
where  $(V_n)$ is a sequence of independent variables such that $V_{n+1}\sim \mu_n$, conditionally on {$\sigma(\mathcal{F}_{n},U_{n+1})$}. Clearly, $(X_n)$ satisfies (\ref{dynamicsxn}) and the  times at which $U_n=\cim$ are the "resurrection"  times (at which  $X_n$ is redistributed).

\begin{theo}[Extinction rate estimation]
\label{th:extinction}
Suppose that the assumptions of Theorem \ref{theo1} are satisfied. Then,
$$\widebar{\theta}_n:=\frac{1}{n}\sum_{k=1}^n \1_{\{U_k=\cim\}}\xrightarrow{n\rightarrow+\infty} 1- \Theta(\mu^\star).$$
\end{theo}
\begin{proof}
 Since $\PE(U_{n+1}=\partial |{\cal F}_n)=\delta(X_n)$ , we can decompose $\widebar{\theta}_n$ as
$$\widebar{\theta}_n=\frac{M_n}{n}+{\mu}_{n}(\delta) $$
where $(M_n)$ is the martingale defined by $M_n=\sum_{k=1}^n (1_{\{U_k=\cim\}}-\delta(X_{k-1})).$ Since the increments of $(M_n)$ are uniformly bounded, $\langle M\rangle_n\le C n$ and it  follows from the strong law of large numbers for martingales, that $\frac{M_n}{n}\rightarrow 0$ $a.s.$ as $n\rightarrow+\infty$.
On the other hand, ${\mu}_{n}(\delta)\xrightarrow{n\rightarrow+\infty}\mu^\star(\delta)= 1-
\Theta (\mu^\star) \ a.s.$ This ends the proof.
\end{proof}

\section{Examples and applications}
\label{sect:appli}

\subsection{Finite Markov Chains}
\label{sec:finite}
In this entire subsection, we consider the simple situation where   $\MM$ is a finite set in order to discuss our main assumptions.

We use the notation $$K(x,y) = K(x,\{y\}),\quad \forall \;x,y \in \MM.$$

We say that $x$ leads to $y$, written $x \hookrightarrow y$, if $\sum_{n \geq 0} K^n(x,y) > 0.$ If $A, B \subset \MM$ we write $A \hookrightarrow B$ whenever there exist $x \in A$ and $y \in B$ such that $x \hookrightarrow y.$

Kernel $K$ is called {\em indecomposable} if there exists $x_0 \in \MM$ such that $x \hookrightarrow x_0$ for all $x \in \MM,$ and {\em irreducible} if $x \hookrightarrow y$ for all $x, y \in \MM.$

Note that Hypothesis $\mathbf{H_1}$ is automatically satisfied (endow $\MM$ with the discrete topology) and that $\mathbf{H_3}$ is equivalent to indecomposability (choose $U = \{x_0\}$ and $\Psi = \delta_{x_0}$). From now on, we investigate separately the irreducible and non-irreducible cases.\medskip

\noindent \textbf{Irreducible setting.} When  $K$ is irreducible Hypothesis $\mathbf{H_4}$ holds with $C(n) = c > 0.$ This follows from the following lemma.
\begin{lem}
\label{lem:hypo4fini}
There exists $c > 0$ such that
$$K^n \1(x)=K^n(x,\MM) \geq c K^n (y, \MM)=c K^n \1(y)$$ for all $n \in \N$ and $x,y \in \MM$ such that $x \hookrightarrow y.$
\end{lem}
\begin{proof} Let $c_1 = \min \{K(x,y) : K(x,y) > 0\}.$ \tcr{If $x \hookrightarrow y,$ then the path which links $x$ and $y$ has at most $|\MM|-1$ steps and hence,
$$\exists r_y\in\{1,\ldots,|\MM|-1\}\quad\textnormal{such that}\quad K^{r_y}(x,y)\ge c_1^{r_y}\ge c_1^{|\MM|-1}.$$ 
From the relation
$$K^n(x,\MM) \geq K^{n+r_y}(x,\MM) = \sum_y K^{r_y}(x,y) K^n(y,\MM)$$
it comes that $$ K^n(x,\MM) \geq c_1^{|\MM|-1} K^n(y,\MM).$$ This proves the result with $c =c_1^{|\MM|-1}.$} 
\end{proof}
{As a consequence, except for the rate of convergence,  we retrieve \cite[Theorem 1.2]{BC15} }({see also \cite{AFP,BGZ} for the convergence result in the case $\gamma_n=\frac{1}{n+1}$).}
\begin{theo}
Suppose $K$ is irreducible and $K(x_0,\MM) < 1$ for some $x_0 \in \MM.$ Then $K$ has a unique QSD $\mu^\star$ and under Hypothesis \ref{hypogain}, $(\mu_n)$ converges almost surely to $\mu^\star.$
\end{theo}
\label{sect:exe}
\subsection*{Bottleneck effect and  condition $\mathbf{H_4}$}
Here we  discuss an example demonstrating the necessity of condition $\mathbf{H_4}$ for non irreducible chains.
 Note that this example  can also be understood as a benchmark of
 more general processes admitting several QSDs such as  general indecomposable  Markov chains.

Suppose $\MM = \MM_1 \cup \MM_2$ where $\MM_1$ and $\MM_2$ are nonempty disjoint sets such that
\begin{enumerate}
\item $\forall x,y \in \MM_i \: x \hookrightarrow y;$
\item $\MM_1 \hookrightarrow \MM_2;$
\item $\MM_2 \not \hookrightarrow \MM_1;$
\item $\MM_2 \hookrightarrow \cim,$ (that is $\exists x \in \MM_2 \:  K(x,\MM) < 1$) and $\MM_1 \not \hookrightarrow \cim.$
\end{enumerate}
Let $K_i$ be the kernel $K$ restricted to $\MM_i.$ That is  $$K_i  = (K(x,y))_{x, y \in \MM_i}.$$ Let $\mu^\star_i$ be the (unique) QSD of $K_i$ and $\Theta_i$ the associated extinction rate. Note that, by irreducibility of $K_i,$ and the Perron Frobenius Theorem, $\Theta_i$ is nothing but the spectral radius of $K_i.$

We consider $\mu^\star_i$ as an element of ${\cal P}(\MM)$ by identifying ${\cal P}(\MM_i)$ with the set of $\mu \in {\cal P}(\MM)$ supported by $\MM_i.$

As previously noticed, $\mathbf{H_1}, \mathbf{H_2}$ and $\mathbf{H_3}$ are always true. However, assumption $\mathbf{H_4}$ might fail to hold. More precisely, we have the following result.
\begin{prop}[Sharpness of $\mathbf{H_4}$]
\label{th:sharpH4}
Condition $\mathbf{H_4}$ holds if and only if $\Theta_1 \leq \Theta_2.$ In this case the unique QSD of $K$ is $\mu^\star_2$ and, under Hypothesis \ref{hypogain} $\mu_n \rightarrow \mu^\star_2.$
\end{prop}

\begin{proof} Fix $x_0 \in \MM_2$ and let $\Psi = \delta_{x_0}.$ Hence, ${\bf H_1}$, ${\bf H_2}$ and  ${\bf H_3}$ hold.
By Lemma \ref{lem:hypo4fini} there exists $c > 0$ such that $\Psi (K^n \1) = K^n \1 (x_0) \geq c K^n \1(x)$ for all $x \in \MM_2$ and $n\ge0$. Thus $\mathbf{H_4}$ is equivalent to
\begin{equation}
\label{eq:H4finite}
\forall x   \in \MM_1, \quad K^n \1 (x_0) \geq C(n) K^n \1 (x)
\end{equation}
{with $C$ satisfying \eqref{eq:condC1}}. Let $\tau_1 = \min \{n \geq 0 \: : Y_n \not \in \MM_1\}$ and $\tau_2 = \min \{n \geq 0 \: : Y_n = \cim \}.$
By Lemma \ref{lem:hypo4fini} applied to each of the kernel $K_i$, and from the relation  $\Theta_i^n = \mu_i K_i^n \1_{\MM_i},$  we get that
for all $x \in \MM_i$
\begin{equation}
\label{eq:uniformtheta}
\frac{1}{c} \Theta_i^n \geq \PE_x(\tau_i > n) \geq c \Theta_i^n
\end{equation} for some $c > 0.$
Thus for all $x \in \MM_1$,
$$\PE_x(\tau_2 > n) = \E_x\left[\PE(\tau_2 > n| {\cal F}_{\tau_1})\right] = \E_x \left[\PE_{Y_{\tau_1}}(\tau_2 > n -\tau_1)\right] $$
$$ \leq \frac{1}{c} \E_{x} \left[\Theta_2^{n-\tau_1} \1_{\tau_1 \leq n} + \1_{\tau_1 > n}\right]$$
by (\ref{eq:uniformtheta}). Thus,
$$\PE_x(\tau_2 > n) \leq \frac{1}{c} \Theta_2^n  \sum_{k = 1}^n \Theta_2^{-k} (\PE_{x}(\tau_1 > k-1) - \PE_{x}(\tau_1 > k))  + \frac{1}{c} \PE_{x}(\tau_1 > n)$$
 $$= \frac{1}{c} \Theta_2^n (\Theta_2^{-1} + \sum_{k = 1}^{n-1} (\Theta_2^{-k-1} - \Theta_2^{-k}) \PE_{x}(\tau_1 > k)) $$
 $$= \frac{1}{c} \Theta_2^{n-1} (1 + (1-\Theta_2) \sum_{k = 1}^{n-1} \Theta_2^{-k} \PE_{x}(\tau_1 > k))$$
 Then,  by (\ref{eq:uniformtheta}) again, we get
 $$\PE_x(\tau_2 > n)  \leq O( \Theta_2^{n-1}) =  O( \PE_{x_0} (\tau_2 > n))$$ when $\Theta_1 < \Theta_2$ and
 $$\PE_x(\tau_2 > n)  \leq O( \Theta_2^{n-1} (1 + n)) = O(  \PE_{x_0} (\tau_2 > n) (1 + n))$$ when $\Theta_2 = \Theta_1.$ This proves that (\ref{eq:H4finite})  holds with $C(t) = C$  when $\Theta_1 < \Theta_2$ and $C(t) = C/(1+t)$ when $\Theta_1 = \Theta_2.$ If now $\Theta_1 > \Theta_2,$ it follows from Theorem \ref{th:2points} below that  another QSD $\mu^\star \neq \mu^\star_2$ exists, hence $\mathbf{H_4}$ fails.

\end{proof}
 For $\Theta_1 \leq \Theta_2,$ $\mu_2^\star$ is a global attractor of the dynamics induced by (\ref{eq:ODEE}), but
  when $\Theta_1$ exceeds $\Theta_2$  a {\em transcritical bifurcation} occurs:   $\mu_2^\star$ becomes a saddle point whose stable manifold is ${\cal P}(\MM_2),$  while there is another linearly stable point $\mu^\star$  whose basin of attraction  is ${\cal P}(\MM) \setminus {\cal P}(\MM_2).$

  This behavior will be shown in section \ref{sec:PMR} and combined with standard techniques from stochastic approximation, it will be used to prove the following result.
\begin{theo}[{Behavior of the algorithm without Assumption $\mathbf{H_4}$}]
\label{th:2points}
 Suppose $\Theta_1 > \Theta_2.$ Then there is another QSD  $\mu^\star$ having full support (i.e~$\mu^\star(x) > 0$ for all $x \in \MM$). Under  Hypothesis \ref{hypogain},
\begin{itemize}
\item[(i)]  $(\mu_n)_{n\geq 0}$ converges almost surely to $\mu_\infty\in\{\mu_2^\star,\mu^\star\}$.
\item[(ii)] If $X_0 \in \MM_2,$ $X_n \in \MM_2$ for all $n$ and {$\mu_\infty = \mu_2^\star$} with probability one.
\item[(iii)] If $X_0 \in \MM_1,$ the event  $\{\mu_\infty=\mu^\star\}$ has positive probability.
\item[(iv)] If $\sum_n\prod_{i=1}^n (1-\gamma_i)<+\infty$,  the event $\{ \exists N \in \N : X_n \in \MM_2\, \mbox{ for all } n \geq N \}$ has positive probability, and on this event $\mu_\infty=\mu_2^\star.$
\end{itemize}
\end{theo}
\begin{exe}[{Two points space}]{\rm
\label{ex:2pt}
The previous results are in particular adapted to the case where $\MM_i = \{i\}, i = 1, 2$ and
$${K = \left(
        \begin{array}{cc}
         a & 1- a\\
          0 & b \\
        \end{array}
      \right)}$$
      with $a,b\in(0,1)$.
      Write $\mu \in {\cal P}(\MM)$ as $\mu = (x,1-x), 0 \leq x \leq 1.$ Then
      $$K_{\mu} = \left(
        \begin{array}{cc}
         a  &  1-a\\
         (1-b) x & b + (1-b) (1-x) \\
        \end{array}
      \right)$$
   and the ODE
      (\ref{eq:ODEE}) writes
 \begin{equation}\label{EDO:2points}
 \dot{x} = - x + \frac{(1-b)x}{(1-a) + (1-b) x}.
 \end{equation}

{In this case, one can check that $\Theta_1=a$ and $\Theta_2=b$, $\mu_2^\star=\delta_2$ and  when $a>b$, $\mu^\star=\frac{a-b}{1-b}\delta_1+\frac{1-a}{1-b}\delta_2$.
In Figure \ref{bifurcationfig},  for a fixed value of $b$, we draw  the phase portrait of the ODE \eqref{EDO:2points} in terms of $a$ and especially the bifurcation which appears when $a>b$. }

  \begin{figure}[h]
\begin{center}
\includegraphics[scale=0.4]{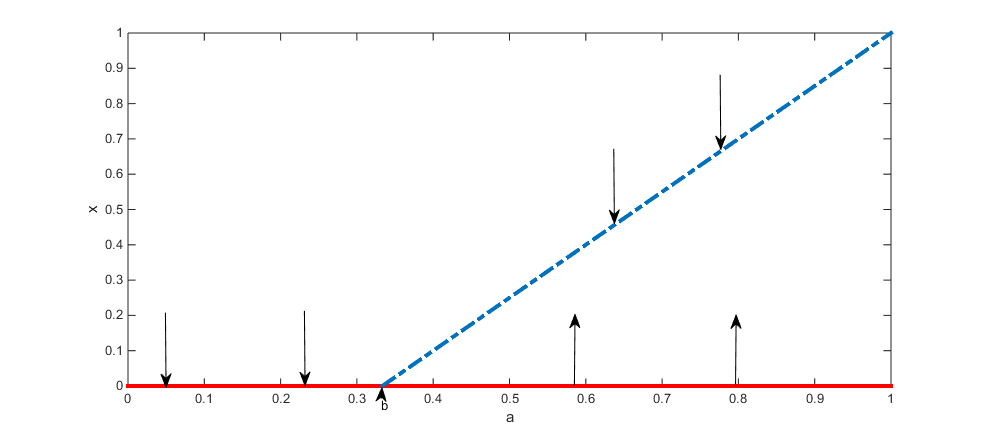}
\caption{\label{bifurcationfig} Transcritical bifurcation associated to Equation \eqref{EDO:2points}; $b=1/3$, Continuous line: $a\mapsto \mu^\star_2(1)$, dotted line: $a\mapsto \mu^\star(1)$ .}\end{center}
\end{figure}

      }
\end{exe}

\begin{Rq}[{Open problem}]{\rm
Suppose $\gamma_n = \frac{A}{n}.$ Although $\mu^\star_2$ is a saddle point when $\Theta_1 > \Theta_2,$ Theorem \ref{th:2points} shows that the event $\mu_n \rightarrow \mu_2^\star$ has positive probability when $A > 1.$   A challenging question would be to prove (or disprove) that this event has zero probability when $A \leq 1.$ This is reminiscent of the situation thoroughly analyzed for two-armed bandit problems
in \cite{LP08,LPT}}.
\end{Rq}

{
\begin{Rq}[Conditioned dynamics]{\rm
Note that by mimicking the proof of Lemma \ref{lem:ode2points} below, one is also able to compute the limit of the conditioned dynamics:
$$
\lim_{n \to \infty} \mathbb{P}_y (Y_n \in \cdot | Y_n \in \MM) = \lim_{n \to \infty} \frac{K^n(y, \cdot)}{K^n \1 (y)}  =\nu,
$$
where $\nu=\mu_2^\star$ if $\Theta_1 \leq \Theta_2$ or $y\in{\cal E}_2$ and $\nu=\mu^\star$ when $\Theta_1 > \Theta_2$ and $y\in{\cal E}_1$. Furthermore, at least for Example \ref{ex:2pt} above, it is worth noting that the convergence is not exponential when $\Theta_1=\Theta_2$.
}
\end{Rq}
}
{
\begin{Rq}[{Fleming-Viot algorithm}]{\rm
{
Theorem \ref{th:2points} shows that, with positive probability,  our algorithm asymptotically matches with  the behavior of the dynamics conditioned to the non-absorption. Surprisingly, this is not the case for the discrete-time (or continuous-time) Fleming-Viot particle system (see \cite[Section 3]{BC15}, for the definition) which always converges to $\mu_2^\star$. Actually, let us recall that this algorithm has two parameters: the (current) time $t\geq 0$ and the number of particles $N\geq 1$. When the number of particles goes to infinity, it is known that the empirical measure (induced by the particles at a fixed time) converges to the laws conditioned to not be killed; see for instance \cite{CT13,V11} in the continuous-time setting. However, if we keep constant the number of particles and  let first the time $t$ tend to infinity then,  one obtains the convergence to $\mu_2^\star$ in place of $\mu^\star$. This comes from the fact that  the state where all the particles are in ${\cal E}_2$ is absorbing and accessible. In this case, the commutation of the limits established in \cite[Section 3]{BC15} fails. Finally, note that the study of the rate of convergence of Fleming-Viot processes in a two-points space is investigated in \cite{CT16}.
}}
\end{Rq}
}
\smallskip

 \subsection{Approximation of QSD of diffusions}
A potential application of this work is to generate a way to simulate QSD of continuous-time Markov dynamics. {To this end, the natural idea is to apply the procedure to a  discretized version (Euler scheme in the sequel) of the process.} Here, we focus on the case of non-degenerate diffusions $(\XX_t)_{t\ge 0}$ in $\ER^d$ killed when leaving a bounded connected open  set $D$.
More precisely, let $(\XX_t)_{t\ge 0}$ be the unique solution to the $d$-dimensional SDE 
 $$d\XX_t=b(\XX_t)dt+\sigma(\XX_t) dW_t, \quad \XX_0 \in D,$$
 where $b$ and $\sigma$ are  defined on $\ER^d$ with values in $\ER^d$  and $\mathbb{M}_{d,d}$ respectively. One assumes below that the diffusion is uniformly elliptic and that $b$ and $\sigma$ belong to ${\cal C}^2(\ER^d)$ (see Remark \ref{Unique_QSD}).\smallskip

 For a given step $h>0$, we denote  by  $(\xi_{t}^h)_{t\ge0}$, the stepwise constant Euler scheme defined by $\xi_0=y\in D$,
 $$\forall n\in \mathbb{N},\quad \xi_{(n+1) h}^h=\xi_{nh}^h+ h b(\xi_{nh}^h)+\sigma(\xi_{nh}^h)(W_{(n+1)h}-W_{nh}),$$
 and for all $t\in[nh, (n+1)h)$, $\xi_t^h=\xi_{nh}^h$. Under the ellipticity assumption  on the diffusion, the Markov chain $(Y_{n}:=\xi_{nh}^h)_n$ satisfies the assumptions of Theorem \ref{theo1} (with ${\MM}=\widebar{D}$) (in particular \eqref{eq:minmaj}) and thus admits a unique QSD  that we denote by $\mu_h^\star$.\smallskip

\noindent   This QSD can be approximated through the procedure defined above and the  natural question is: does $(\mu_h^\star)_h$ converge to $\mu^\star$ when $h\rightarrow0$, where $\mu^\star$ denotes the unique QSD of $(\XX_t)_{t\ge 0}$ killed when leaving $D$ ? A positive answer is given below.

 \begin{theo}[Euler scheme approximation]\label{prop:diff1} Assume that $(\XX_t)_{t\ge0}$ is a uniformly elliptic diffusion and that $D$ is a bounded domain (i.e. connected open set)  with ${\cal C}^{3}$-boundary. Then, $((\XX_t)_{t\ge0},\partial D)$
 admits a unique QSD $\mu^\star$ and  $(\mu_h^\star)_{h>0}$ converges weakly to $\mu^\star$ when $h\rightarrow0$.
 \end{theo}
\begin{Rq}[{Smoothness assumptions}]\label{Unique_QSD}
The uniqueness of $\mu^\star$ is given by Theorem 5.5 of  Chapter $3$ in \cite{P95} (see also \cite[Theorem (C)]{GongZhao}). Also note that in the proof of the above theorem, one makes use of some results of \cite{gobet} about the discretization of killed diffusions. The ${\cal C}^2$-assumption on $b$ and $\sigma$ is adapted to the setting of {these} papers but could  be probably relaxed in our context.
\end{Rq}

We propose to illustrate the previous results by some simulations.
We consider an Ornstein-Uhlenbeck process
$$d\XX_t=-\XX_t dt+ dB_t,\quad t\ge0$$
killed outside an interval $[a,b]$ and thus compute the sequence $(\mu_n^h)_{n\ge1}$ with step $h$.

We will assume that $a=0$ and $b=3$. In Figure \ref{fig:OU}, we represent on the left the approximated density of $\mu_n^h$ (obtained by a convolution with a Gaussian kernel) for a fixed value of  $h$ and different values of $n$. Then, on the right, $n$ is fixed $(n=10^7)$ and $h$ decreases to $0$.
\begin{figure}[htbp]
\hspace{-0cm}
   \begin{minipage}[c]{.46\linewidth}
   \includegraphics[scale=0.5]{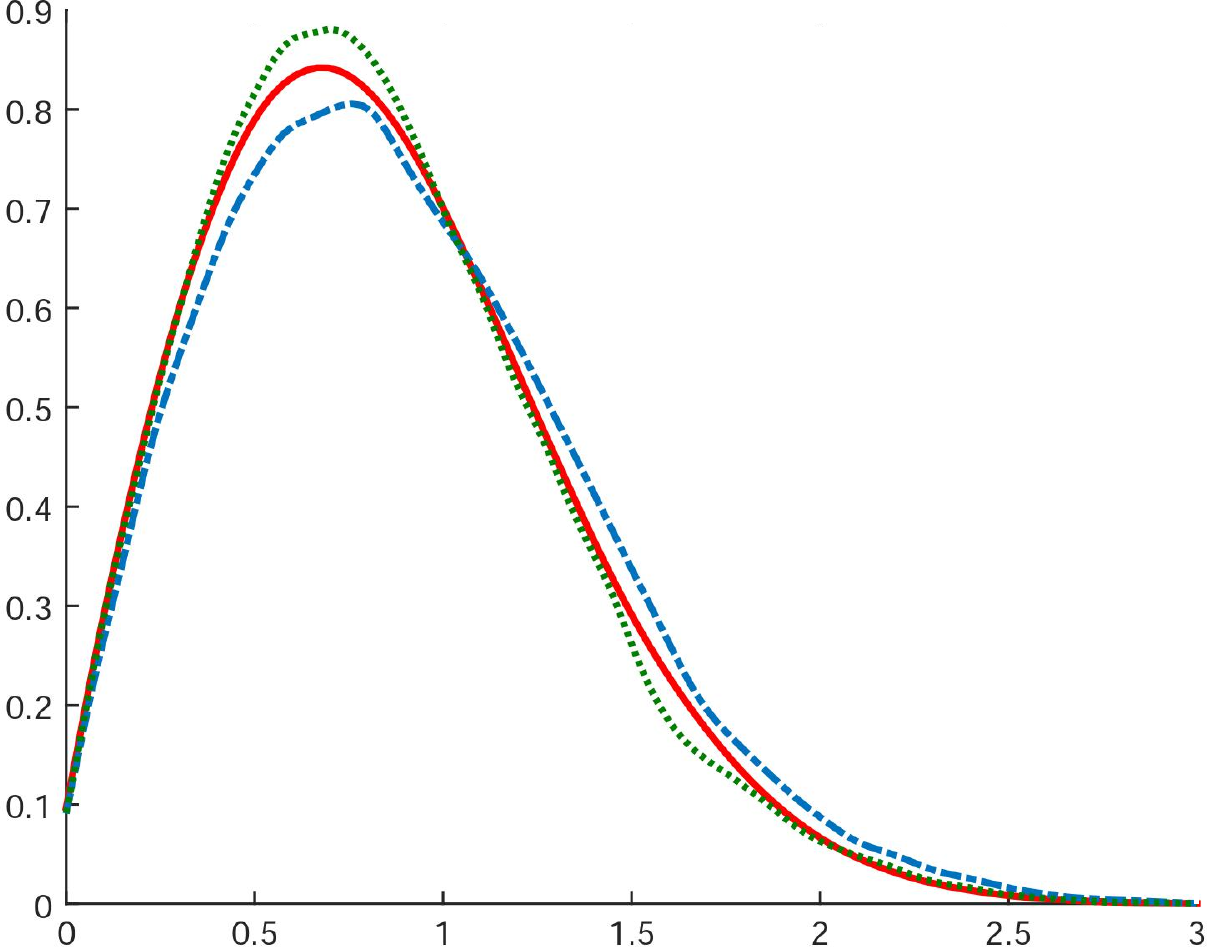}
   \end{minipage} \hfill
   \hspace{-4cm}
   \begin{minipage}[c]{.46\linewidth}
     \includegraphics[scale=0.5]{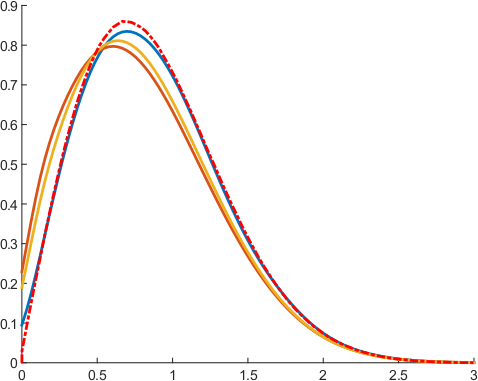}
   \end{minipage}
   \caption{\label{fig:OU} Left: Approximated density of $(\mu_n^h)$ with $h=0.01$ and $n=5.10^4,10^5,10^6$ (green, blue, red) Right: Comparison of $\mu^\star_h$ for $h=0.05,0.01, 0.001$, (red, orange, blue)
   with $\mu^\star$ (red, dotted-line) }
\end{figure}
Unfortunately, even though the convergence in $n$ seems to be fast, the convergence of $\mu_h^\star$ towards $\mu^\star$ is very slow: the  discretization of the problem {underestimates} the probability to be killed between two discretization times. The slow convergence means in fact that this probability decreases slowly to $0$ with $h$. \smallskip

\noindent However, it is now well-known that, under some conditions on the domain and/or on the dimension, it is possible to compute a sharp estimate of this probability. More precisely, let  ${(\widetilde{\xi}_t^h})_t$
denote the refined continuous-time Euler scheme $\widetilde{\xi}_{nh}^h=\xi_{nh}^h$ and for all $t\in[nh, (n+1)h)$,
$$ \widetilde{\xi}_t^h=\widetilde{\xi}_{nh}^h+ (t-nh) b(\widetilde{\xi}_{nh}^h)+\sigma(\widetilde{\xi}_{nh}^h) (W_t-W_{nh}).$$
It can be shown that
$${\cal L}\left((\widetilde{\xi}_t^h)_{t\in[nh,(n+1)h]}|\widetilde{\xi}_{nh}^h=x,\widetilde{\xi}_{(n+1)h}^h=y\right)={\cal L}\left(x+\frac{t-nh}{h}(y-x)+\sigma(\widetilde{\xi}_{nh}^h) {\bf B}_t^h\right)$$
where for a given $T>0$, ${\bf B}^T$ denotes the Brownian Bridge on the interval $[0,T]$ defined by: $ {\bf B}_t^T= W_t-\frac{t}{T}W_T$. In dimension $1$, the law of the infimum and the supremum of
the Brownian Bridge can be computed (see \cite{gobet} for details and a discussion about higher dimension). One has for every $z\ge \max(x,y)$,
$$\PE\left(\sup_{t\in[0,T]} \left(x+(y-x)\frac{t}{T} +\lambda {\bf B}^{T}\right)\le z \right) =1-\exp\left(-\frac{2}{T\lambda^2}(z-x)(z-y)\right).$$
Thus, this means that at each step $n$, if $\xi_{(n+1)h}\in D$, one  can compute, with the help of the above properties, a Bernoulli random variable $V$ with parameter
$$p=\PE( \exists t\in(nh,(n+1)h), \xi_t^h\in D^c|\xi_{nh}=x,\xi_{(n+1)h}=y)\quad\textnormal{ (If $V=1$, the particle is killed)}.$$
 This refined algorithm has been tested numerically and illustrated in Figure \ref{fig2:OU}.

\begin{figure}[h]
\begin{center}
\vspace{-4cm}
\includegraphics[scale=0.5]{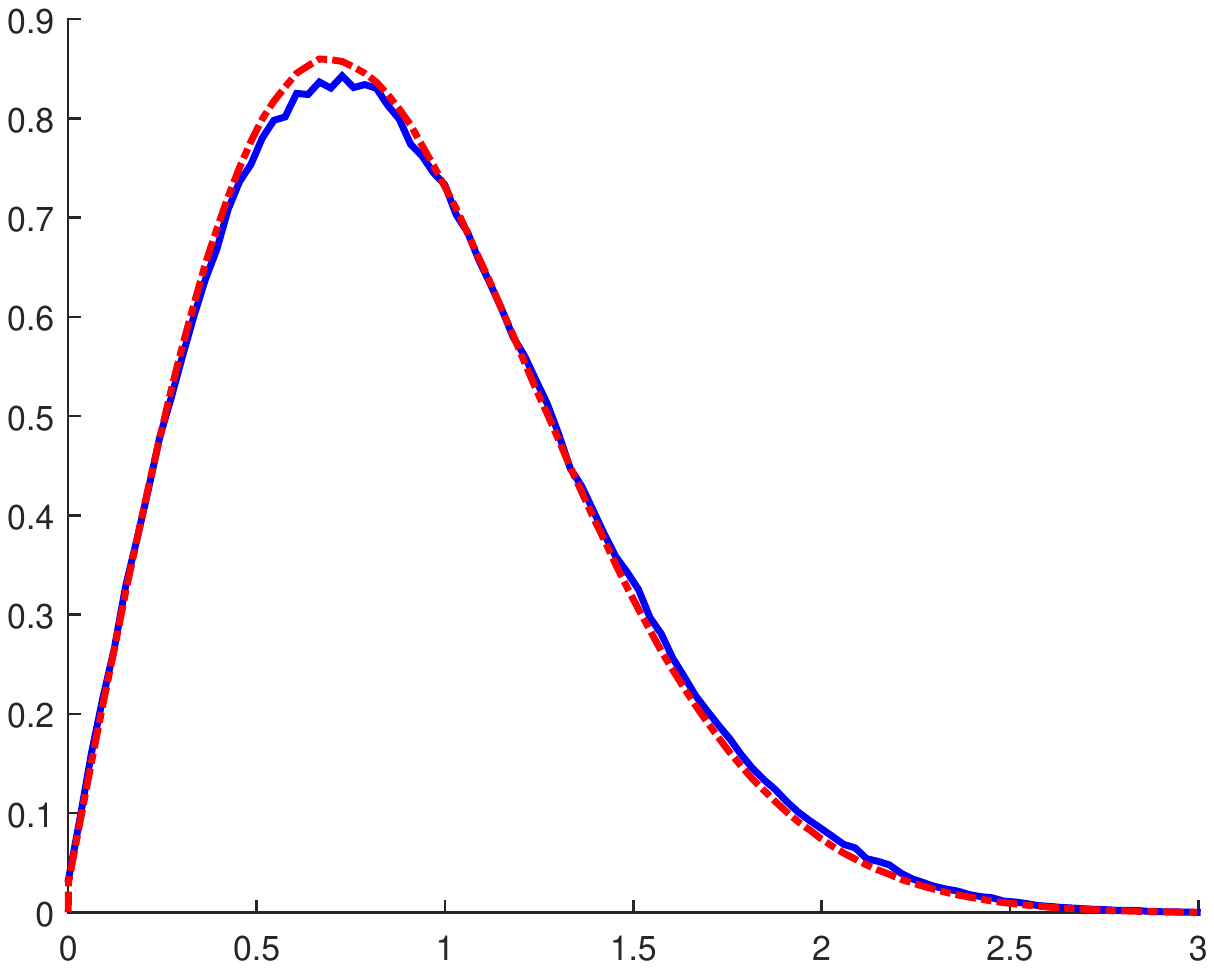}
\vspace{-4cm}
\caption{\label{fig2:OU} Approximation of $\mu^\star_h$ with the Brownian Bridge method
for $h=0.1$ (blue) compared with the reference density (red, dotted-line)}
\end{center}
\end{figure}

{
\begin{Rq} \label{rq:decreasstep} Here, the effect of the Brownian Bridge method is only considered from a numerical viewpoint. The theoretical consequences on the rate of convergence are outside of the scope of this paper.
Also remark that in order to get only one asymptotic for the algorithm, it would be natural to replace the constant step $h$ by a decreasing sequence as in \cite{L07,P08}. Once again, such a theoretical extension is left to a future work.\end{Rq}
}

\noindent \textbf{Outline of the proofs.}  In Section \ref{sect:Prelim}, we begin by some preliminaries: the starting point is to show that the QSD is a fixed point for the application $\mu\mapsto\Pi_\mu$ (on ${\cal P}({\MM})$) where $\Pi_\mu$ denotes the invariant distribution of $K_\mu$ (see Lemma \ref{lemme1} below). Then,  in order to give a rigorous sense to the ODE \eqref{eq:ODEE}, we prove that this application is Lipschitz continuous for the total variation norm (Proposition \ref{lem:explicitpimu}) by taking advantage of the exponential ergodicity of the transition kernel $K_\mu$  and the control of the exit time $\tau$ (see Lemma \ref{lem:expect} and Lemma \ref{lem3} ). In Section \ref{sect:ODE}, {we define the solution of the ODE and prove its global asymptotic stability}. In Section \ref{sect:APT}, we then show that {(a scaled version of)} $(\mu_n)_{n\ge0}$ is an asymptotic  pseudo-trajectory for the ODE.
The proofs of Theorems  \ref{theo1} and \ref{th:cvmarche} are finally achieved at the beginning of Section \ref{sec:PMR}. In this section, we also prove the main results of Section 3: Theorems  \ref{th:2points} and \ref{prop:diff1}. We end the paper by some possible  extensions of our present work.\smallskip

\section{Preliminaries}
\label{sect:Prelim}


We begin the proof by a series of preliminary lemmas.
The first one provides uniform estimates on the extinction time
\begin{equation}
\label{eq:deftau}
\tau = \min\{ n \geq 0 : Y_n = \cim\}
 \end{equation}
  \tcr{where  $(Y_n)_{n\ge0}$ is a Markov chain with transition  $\widebar{K}$ defined in \eqref{eq:defkbar}}.
\begin{lem}[Expectation of the extinction time] \label{lem:expect}
Assume $\mathbf{H_1}$ and $\mathbf{H_2}.$ Then \begin{itemize}
\item[(i)] There exist $N \in \N$ and $\delta_0 > 0$ such that for all $x\in {\MM}$,  $\widebar{K}^N(x,\{\cim\})\geq \delta_0$.
\item[(ii)]
$$\sup_{x\in{\MM}}\ES_x[\tau]<+\infty.$$
\end{itemize}
\end{lem}
\begin{proof}
\textit{(i)} By $\mathbf{H_1}$ the map $x  \mapsto \widebar{K}^N(x,\cim)=1-{K}^N \1(x)$ is continuous on $\MM.$ It then suffices to show that there exists $N\in\N$ such that $\widebar{K}^N(x,\cim)>0$ for all $x\in{\cal E}$.
Suppose to the contrary that for all $N \in \N$ there exists $x_N \in \MM$ such that $\widebar{K}^N(x_N,\cim) =  0.$ Hence $\widebar{K}^k(x_N,\cim) =  0$ for all $k \leq N.$
By  compactness of ${\MM}$, we can always   assume (by replacing $(x_N)$ by a subsequence) that $x_N \rightarrow x^*\in{\MM}.$ Thus
 $\widebar{K}^k (x_N,\partial)\xrightarrow{N\rightarrow \infty} \widebar{K}^k (x^*,\partial) = 0$ for all $k\in\mathbb{N}.$ This leads to a contradiction with assumption $\mathbf{H_2}.$

\textit{(ii)} Let $N$ and $\delta_0$ be like in $(i).$  By the Markov property, for all $k\in\mathbb{N}^*$
\begin{equation}\label{eq:queueexpo}
\PE_x(\tau>k N) =\E_x \left[\PE_{Y_{(k-1)N}}(\tau>N) 1_{ \tau>(k-1)N} \right] \le (1-\delta_0) \PE_x(\tau> (k-1)N).
\end{equation}
Thus, for all $k \in \N$ $$\PE_x(\tau>kN)\le (1-\delta_0)^k$$ and, consequently, $$\frac{1}{N}\ES_x[\tau] \leq \frac{1}{\delta_0} + 1.$$
\end{proof}
\tcr{\begin{Rq} Note that \eqref{eq:queueexpo} leads in fact to the following statement: there exists $\lambda>0$ such that $\sup_{x\in\MM}\ES[e^{\lambda \tau}]<+\infty$.
\end{Rq}}
\tcr{The following lemma is reminiscent of the approach developed in \cite{Ferrari} for Markov chains on the positive integers killed at the origin  and in \cite{BenAri} for diffusions killed on the boundary of a domain.}
\begin{lem}[Invariant distributions and QSD]\label{lemme1} Assume $\mathbf{H_1}$. Then,
\begin{itemize}
\item[(i)] For every $\mu\in{\cal P}({\MM})$, $K_\mu$ is a Feller kernel  and admits at least one invariant probability.
\item[(ii)] A probability  $\mu^\star$ is a QSD for $K$ if and only if it is an invariant probability of $K_{\mu^\star}$.
\item[(iii)] Assume {that for every $\mu$,}  $K_{\mu}$ has a unique invariant probability  $\Pi_\mu$. Then $\mu \mapsto \Pi_\mu$ is continuous in ${\cal P}(\MM)$ (i.e~for the topology of weak convergence)  and then there exists $\mu^\star\in{\cal P}({\MM})$ such that $\mu^\star=\Pi_{\mu^{\star}}$ or, equivalently, a QSD $\mu^\star$ for $K.$
\end{itemize}
\end{lem}

\begin{proof} \textit{(i)} The Feller property is obvious under $\mathbf{H_1}$ and it is well known that a Feller Markov chain on a compact space has an invariant probability (since any weak limit of the sequence  $(\frac{1}{n} \sum_{k = 1}^n \nu K^n_{\mu})_{n \geq 0}$ is an invariant probability). \smallskip

\noindent  \textit{(ii)} Since $\delta =\1 -K\1$, for every $A\in {\cal B}({\MM})$, we have
$$\mu^\star (A) = (\mu^\star K_{\mu^\star})(A) \Leftrightarrow \mu^\star(A)=\frac{(\mu^\star K)(A)}{(\mu^\star K) \1}.$$
But, by definition $\mu^\star$ is a QSD if  and only if the right-hand side is satisfied for  every $A\in {\cal B}({\MM})$.

\noindent  \textit{(iii)} Let $(\mu_n)_{n\geq 0}$ be a probability sequence converging  to some $\mu$ in ${\cal P}(\MM).$ Replacing $(\mu_n)_{n\geq 0}$ by a subsequence,  we can always assume, by compactness of ${\cal P}(\MM),$  that  $(\Pi_{\mu_n})_{n\geq 0}$ converges to some $\nu.$  For every $n\geq 0$ and $f \in {\cal C}(\MM,\R)$, we have
$$
\Pi_{\mu_n}(f) =\Pi_{\mu_n}(K_{\mu_n} f)= \Pi_{\mu_n} (Kf) + \Pi_{\mu_n}(\delta) \mu_n(f).
$$
By $\mathbf{H_1}$, the maps $Kf$ and $\delta$ are continuous and hence by letting $n \to \infty$, one obtains
$$
\nu(f) = \nu(Kf)+ \nu(\delta) \mu(f),
$$
namely $\nu$ is an invariant for $K_\mu$. By uniqueness $\nu = \Pi_\mu$. This proves the continuity of the map $\mu \mapsto \Pi_{\mu}.$
Now, since ${\cal P}({\MM})$ is a convex compact subset of a locally convex topological space (the space of signed measures equipped with the weak* topology)  every continuous mapping from ${\cal P}({\MM})$ into itself has a fixed point by Leray-Schauder-Tychonoff fixed point theorem.

\end{proof}
For all $\mu \in {\cal P}(M)$ and $t \geq 0$ we let $P_t^{\mu}$ denote the Markov kernel on $\MM$ defined by
 \begin{equation}\label{poissonization}
 {P}_t^{\mu}(x,\cdot)
:= e^{-t} \sum_n \frac{t^n}{n!} K_{\mu}^n(x,\cdot).
\end{equation}
It is classical (and easy to verify) that
\begin{itemize}
\item[(a)] $(P_t^{\mu})_{t \geq 0}$ is a semigroup (i.e $P_{t+s}^{\mu} f = P_t^{\mu} P_s^{\mu} f$ for all $f \in {\cal B}(\MM,\R)$);
    \item[(b)] Every invariant probability for $K_{\mu}$ is invariant for $P_t^{\mu};$
    \item[(c)] $P_t^{\mu}$ is Feller whenever $K_{\mu}$ is (in particular under $\mathbf{H_1}$).
    \end{itemize}
\smallskip
{ If $(X_n^\mu)_{n\ge0}$ is  a Markov chain with transition $K_\mu$, $(P_t^{\mu})_{t\geq0}$ denotes the semi-group of $(X_{N_t}^\mu)_{t\ge0}$ where $(N_t)_{t\ge0}$ is an independent Poisson process with intensity $1$.}

For any finite signed measure $\nu$ on $M$ recall that the {\em total variation norm} of $\nu$ is defined
as
\begin{eqnarray}
\label{def:vartot}
\Vert \nu \Vert_{\textrm{TV}}  &=& \sup\{ |\nu f| ~:~ f \in {\cal B}(\MM,\R),  \|f\|_{\infty} \leq 1\}\\
&=& \nu^+(\MM) + \nu^-(\MM) \nonumber
\end{eqnarray}
where $\nu = \nu^+ - \nu^-$ is the Hahn Jordan decomposition of $\nu.$
Let us recall that if {$P$} is a Markov kernel on $M$ and $\alpha, \beta \in {\cal P}(\MM),$ then
{\begin{equation}
\label{eq:contractTV}
\Vert\alpha P - \beta P\Vert_{\textrm{TV}} \leq \Vert\alpha - \beta\Vert_{\textrm{TV}}
\end{equation}
since $\|P f\|_{\infty} \leq \|f\|_{\infty}.$}
\begin{lem}[Uniform exponential ergodicity] \label{lem3} Assume $\mathbf{H_1}$ and $\mathbf{H_2}.$ Then there exists $0 < \varepsilon < 1$ such that for all  $\alpha, \beta, \mu \in{\cal P}({\cal E})$ and  $t \geq 0$
$$ \|\alpha P_t^{\mu}-\beta P_t^{\mu}\|_{\textrm{TV}}\le (1-\varepsilon)^{\lfloor t \rfloor} \|\alpha-\beta\|_{\textrm{TV}}.$$ In particular, {if $\Pi_\mu$ denotes an invariant probability for $K_\mu$},
$$ \|\alpha P_t^{\mu}-\Pi_{\mu}\|_{\textrm{TV}}\le (1-\varepsilon)^{\lfloor t \rfloor} \|\alpha-\Pi_{\mu}\|_{\textrm{TV}}.$$
{As a consequence,  $K_\mu$ has a unique invariant probability.}
\end{lem}
\begin{proof}
\textit{(i)}. Set $P_{\mu} = P_1^{\mu}.$  Let $\delta_0 > 0$ and $N \in \N$ be given by Lemma \ref{lem:expect} \textit{(i)}.
It easily seen by induction  that for all $k \geq 1$ and $f : \MM \mapsto [0,\infty[$ measurable, $$K^k_{\mu} f \geq \mu(f) K^{k-1} \delta.$$
Thus,
\begin{eqnarray}
\label{doeblin}
 \nonumber P_{\mu} f  \geq   \frac{1}{e} \mu(f) \sum_{k = 1}^{N} \frac{1}{k!}  K^{k-1} \delta  \geq \frac{1}{e N!}
    \mu(f) \sum_{k = 1}^N K^{k-1} \delta \\
  =    \frac{1}{e N!}  \mu(f) (\1 - K^N \1)   \geq \varepsilon \mu(f)
\end{eqnarray}
where $$\varepsilon =  \frac{1}{e N!} \delta_0.$$
Let  $R_\mu$ be the kernel on $\MM$ defined by
\begin{equation}
\label{eq:K=psi+R}
\forall x\in {\MM}, \quad P_\mu(x,.)=\varepsilon \mu(.) +(1-\varepsilon) R_\mu(x,.).
\end{equation}
Inequality (\ref{doeblin})  makes $R_{\mu}$ a Markov kernel. Thus for all $\alpha, \beta \in {\cal P}(\MM)$

$$\|\alpha P_\mu-\beta P_\mu\|_{\textrm{TV}} = (1-\varepsilon) \|\alpha R_\mu-\beta R_\mu\|_{\textrm{TV}} \le (1-\varepsilon)\|\alpha-\beta\|_{\textrm{TV}},$$ (where the last inequality follows from (\ref{eq:contractTV}))
and, by induction,
$$\|\alpha P_\mu^n-\beta P_\mu^n\|_{\textrm{TV}}\le (1-\varepsilon)^n\|\alpha-\beta\|_{\textrm{TV}}.$$
Now, for all $t \geq 0$ write $t = n + r$ with $n \in \N$ and $0\leq r < 1.$ Then, {
\begin{align*}
\|\alpha P_t^{\mu}-\beta P_t^\mu \|_{\textrm{TV}} &= \|\alpha  P_r^{\mu} P_{\mu}^n -\beta P_r^\mu P_{\mu}^n \|_{\textrm{TV}}\\
&\le (1-\varepsilon)^{n}\|\alpha P^\mu_{r}-\beta P^\mu_r\|_{\textrm{TV}} \le (1-\varepsilon)^n \|\alpha-\beta\|_{\textrm{TV}}.
\end{align*}}
{As mentioned before, if $\Pi_\mu$ is an invariant probability for $K_\mu$, $\Pi_\mu$ is also an invariant probability for $(P^\mu_t)_{t\geq 0}$. The second inequality is thus obtained by setting $\beta=\Pi_\mu$ and uniqueness of the invariant probability is a consequence of the convergence of $(\alpha P_t^\mu)_{t\ge0}$} towards $\Pi_\mu$.
\end{proof}

\subsection{Explicit form for $\Pi_\mu$.} Let us denote by $ \Aa$ the transition kernel on $\MM$ defined by
$$\Aa(x,.)=\sum_{n\ge0} K^n(x,.)$$
and set
 $$\|\Aa\|_{\infty} = \sup \{ \|\Aa f\|_{\infty} ~:~ f \in {\cal B}(\MM,\R), \|f\|_{\infty} \leq 1 \}.$$
 {Remark that}
  $$\|\Aa\|_{\infty} = \sup_{x \in \MM} \Aa(x,\MM) \in [0, \infty].$$
\begin{prop}  \label{lem:explicitpimu} Assume $\mathbf{H_1}$ and $\mathbf{H_2}$. Then:
\begin{itemize}
\item[(i)] For all  $x\in {\MM},$
$$1 \leq \Aa(x,\MM) {=\Aa\mathbf{1}(x) } = \ES_x[\tau] \leq  \|\Aa\|_{\infty} < \infty.$$
\item[(ii)] For all $\mu \in {\cal P}(M),$
\begin{equation}\label{explicitpimu}
\Pi_\mu=\frac{\mu \Aa}{(\mu \Aa)(\1)}.
\end{equation}
\item[(iii)] The map $\mu \mapsto \Pi_{\mu}$ is Lipschitz continuous for the total variation distance.
\end{itemize}
\end{prop}
\begin{proof} \textit{(i)} The inequality ${\cal A}(x,{\cal E})\ge 1$ is obvious. For the second one, we remark that for all $x \in \MM$
$$\Aa(x,\MM) = \sum_{n\ge0} K^n(x,\MM) = \sum_{n\ge0} \mathbb{P}_x(\tau > n) = \ES_x[\tau] \leq \sup_x \ES_x(\tau) < \infty$$
 where the last inequality follows from {Lemma \ref{lem:expect}}.

 \textit{(ii)}  For any $f \in {\cal B}(\MM,\R),$
$$\mu \Aa K_\mu(f)=\mu\left(\sum_{n\ge0} (K^{n+1} f+K^n\delta \mu(f))\right)=\sum_{n\ge0}\mu K^{n+1}f+\mu(f)\mu(\sum_{n\ge0} K^n(\delta)).$$
Since $\sum_{n\ge0} K^n\delta(x) = \sum_{n\ge0}\left(K^n(x,\MM) - K^{n+1}(x, \MM)\right) = \Aa(x,\MM) - (\Aa(x,\MM) - 1)= 1,$ it follows that
$$(\mu  \Aa ) K_\mu(f)=\mu(f)+\sum_{n\ge1} \mu K^n f=(\mu  \Aa)(f).$$
As a consequence, $\mu  \Aa$ is an invariant measure and it remains to divide by its mass to obtain an invariant probability.

\textit{(iii)} It follows from \textit{(i)} that $\|\mu \Aa\|_{\textrm{TV}} \leq \|\mu\|_{\textrm{TV}} \|\Aa\|_{\infty}$ and $\mu \Aa 1 \geq 1.$ Thus, {reducing the fraction,} it easily follows from \textit{(ii)} that $\|\Pi_{\mu} - \Pi_{\nu}\|_{\textrm{TV}} \leq 2 \|\Aa\|_{\infty} \|\mu - \nu\|_{\textrm{TV}}.$
\end{proof}
\section{The limiting ODE}
\label{sect:ODE}

As mentioned before, the idea of the proof of Theorem \ref{theo1} is to show that the long time behavior of   $(\mu_n)_{n\geq 0}$ can be precisely related to the long term behavior of a deterministic dynamical system  ${\cal P}(\MM)$ induced by the "ODE"
\begin{equation}
\label{eq:ODE}
`` \dot{\mu} = - \mu + \Pi_{\mu}."
\end{equation}
The purpose of this section is to define rigorously this dynamical system and to investigate some of its asymptotic properties.

Throughout the section, hypotheses $\mathbf{H_1}$ and $\mathbf{H_2}$ are implicitly assumed. Recall that ${\cal P}(\MM)$ is a compact metric space equipped with a distance metrizing the weak* convergence.

A {\em semi-flow} on ${\cal P}(\MM)$ is a continuous map $$\Phi : \R^+ \times {\cal P}(\MM) \to {\cal P}(\MM),$$
$$(t,\mu) \mapsto \Phi_{t}(\mu)$$ such that
$$\Phi_0(\mu) = \mu \mbox{ and } \Phi_{t+s}(\mu) = \Phi_t \circ \Phi_s(\mu).$$ We call such a semi-flow  {\em injective} if each of the maps $\Phi_t$ is injective.

A {\em weak solution} to (\ref{eq:ODE}) with initial condition $\mu \in  {\cal P}(\MM),$ is a continuous map  $\xi : \R^+ \mapsto {\cal P}(\MM)$ such that
$$\xi(t) f = \mu f + \int_0^t (-\xi(s) f +  \Pi_{\xi(s)} f) ds$$  for all $f \in {\cal C}(\MM)$ and $t \geq 0.$

We shall now show that there exists an injective semi-flow $\Phi$ on ${\cal P}(\MM)$ such that the trajectory $t \rightarrow \Phi_t(\mu)$ is the unique weak solution to (\ref{eq:ODE}) with initial condition $\mu.$

Let  ${\cal M}_s(\MM)$ be the space of finite signed measures on $\MM$  equipped with the total variation norm $\Vert \cdot \Vert_{\textrm{TV}}$  (defined by  equation (\ref{def:vartot})). By a Riesz type theorem,
 ${\cal M}_s(\MM)$ is a Banach space which can be identified with the dual space of ${\cal C}(\MM,\R)$ equipped with the uniform norm  (see e.g~ \cite[chapter 7]{dudley}). In particular, the supremum in the definition of $\Vert \cdot \Vert_{\textrm{TV}}$  can be taken over continuous functions.

Proposition \ref{lem:explicitpimu} \textit{(i)} and the fact that $K$ is Feller imply  that  $\sum K^n f  $ is normally convergent in ${\cal C}(\MM, \R)$ for any $f\in {\cal C}(\MM, \R)$. More precisely, $\sum_{n\ge0} \| K^n f\|_\infty\le \|{\cal A}\|_\infty\|f\|_\infty $ and hence $f \rightarrow  \Aa f$ is a bounded operator on  ${\cal C}(\MM, \R)$. Furthermore,  its adjoint
$\mu \rightarrow \mu \Aa$ is bounded on  ${\cal M}_s(\MM).$ Thus, by standard results on linear differential equations in Banach spaces,  $e^{t \Aa}$ is a well defined bounded operator
and the mappings $(t,f) \rightarrow  e^{t \Aa} f$ and  $(t,\mu) \rightarrow \mu e^{t \Aa}$ are $C^{\infty}$ mappings satisfying the differential equations
 $$\frac{d}{dt} ( e^{t \Aa} f) =  (e^{t \Aa} \Aa f) = \Aa e^{t \Aa} f$$ and $$\frac{d}{dt} (\mu e^{t \Aa}) = \mu  (e^{t \Aa} \Aa) = \mu \Aa e^{t \Aa}.$$

 For $\mu \in {\cal P}(M)$ and $t \geq 0$ set
 \beq
 \label{eq:gt}
 g_t = e^{t \Aa} \1 \in {\cal C}(\MM),
 \eeq
 $$\widetilde{\Phi}_t(\mu) : = \frac{ \mu e^{t \Aa}}{\mu g_t} \in {\cal P}(\MM),$$ and
 $$s_{\mu}(t) = \int_0^t \widetilde{\Phi}_s(\mu) \Aa \1 ds.$$
 Note that{, by Proposition \ref{lem:explicitpimu} (i), $\dot{s}_{\mu}(t) =   \widetilde{\Phi}_t(\mu) \Aa \1 =\frac{\mu e^{t \Aa}\Aa \1}{\mu e^{t \Aa} \1} \geq 1$} and hence $s_{\mu}$ maps diffeomorphically $\R^+$ onto itself. We let
 $\tau_{\mu}$ denote its inverse and
 \beq
 \label{eq:defPhi}
 \Phi_t(\mu) = \widetilde{\Phi}_{\tau_{\mu}(t)}(\mu)
 \eeq

\begin{prop}
The map $\Phi$ defined by (\ref{eq:defPhi}) is an injective semi-flow  on ${\cal P}(\MM)$ and
for all $\mu \in {\cal P}(\MM),$
$t \mapsto \Phi_t(\mu)$ is the unique  weak solution to (\ref{eq:ODE}) with initial condition $\mu.$
\end{prop}
\begin{proof}
{\textbf{Step 1} (Continuity of $\Phi$) :}
 Let $\mu_n \rightarrow \mu$ in ${\cal P}(\MM)$ and $t_n\rightarrow t.$ Then for all $f \in {\cal C}(\MM)$
$$|\mu_n e^{t_n \Aa} f - \mu e^{t \Aa} f| \leq |\mu_n e^{t_n \Aa} f - \mu_n e^{t \Aa} f| + |\mu_n e^{t \Aa} f -  \mu e^{t \Aa} f|$$
 $$\leq \| e^{t_n \Aa}  -  e^{t \Aa}\|_{\infty} \|f\|_{\infty} + |\mu_n e^{t \Aa} f -  \mu e^{t \Aa} f|.$$ The second term goes to zero because $\mu_n  \rightarrow \mu$ and the first one by strong continuity of $t \mapsto e^{t \Aa}.$ This easily implies that the maps $(t,\mu) \rightarrow \widetilde{\Phi}_t(\mu)$  and  $(t,\mu) \rightarrow s_{\mu}(t)$ are  continuous.
The continuity of the  latter combined with the relation $s_{\mu_n} \circ \tau_{\mu_n}(t_n) = t_n$ implies that every limit point of $\{\tau_{\mu_n}(t_n)\}$ equals $\tau_{\mu}(t);$ but since
$\tau_{\mu}(t) \leq t$ (because $s_{\mu}(t) \geq t$) the sequence $\{\tau_{\mu_n}(t_n)\}$ is bounded and this proves the continuity of $(t,\mu) \rightarrow \tau_{\mu}(t).$ Continuity of $\Phi$ follows.\smallskip

\noindent {\textbf{Step 2}  (Injectivity of $\Phi$):} Suppose $\Phi_t(\mu) = \Phi_t(\nu)$ for some $t \geq 0, \mu, \nu \in {\cal P}(\MM).$ Set $\tau = \tau_{\mu}(t)$ and $\sigma = \tau_{\nu}(t).$ {Assume $\sigma \ge \tau.$ Multiplying the  equality $\widetilde{\Phi}_{\tau}(\mu) = \widetilde{\Phi}_{\sigma}(\nu)$ by $e^{-\tau \Aa}$ shows that $\mu = \widetilde{\Phi}_{\sigma-\tau}(\nu).$
 Thus
\begin{align*}
t&=s_\mu(\tau)=\int_0^\tau \widetilde{\Phi}_{s+\sigma-\tau}(\nu) \Aa 1ds=\int_0^\sigma \widetilde{\Phi}_{s}(\nu) \Aa 1ds - \int_0^{\sigma-\tau} \widetilde{\Phi}_{s}(\nu) \Aa 1ds \\
&=t- s_\nu(\sigma-\tau)
\end{align*}
This implies that $\tau = \sigma,$ hence $\mu = \nu.$}\smallskip

\noindent {\textbf{Step 3} ( $t \rightarrow  \Phi_t(\mu)$ is a weak solution):} The mappings $t \rightarrow \widetilde{\mu}_t  := \widetilde{\Phi}_t(\mu)$ and $t \rightarrow \mu_t  := \Phi_t(\mu)$ are $C^{\infty}$ from $\R^+$ into ${\cal M}_s(M).$ Furthermore,
$$\dot{\widetilde{\mu}}_t = \widetilde{\mu}_t \Aa - (\widetilde{\mu}_t \Aa \1) \widetilde{\mu_t} = \dot{s}_{\mu}(t) (-\widetilde{\mu_t} + \Pi_{\widetilde{\mu_t}}),$$
so that
$$\dot{\mu}_t = - \mu_t + \Pi_{\mu_t}$$ and, in particular,
 $$\mu_t f - \mu_0 f = \int_0^t (-\mu_s f + \Pi_{\mu_s} f) ds$$ for all $f \in {\cal C}(\MM).$\smallskip

\noindent {\textbf{Step 4} (Uniqueness and flow property):}
Let $\{\mu_t\}$ and $\{\nu_t\}$ be two weak solutions of (\ref{eq:ODE}). By separability of ${\cal C}(\MM)$, $\|\mu_t - \nu_t\|_{\textrm{TV}} = \sup_{f \in {\cal H}}  |\mu_t f - \nu_t f|$ for some countable set ${\cal H} \subset {\cal C}(\MM).$ This shows that $t \rightarrow \|\mu_t - \nu_t\|_{\textrm{TV}}$ is measurable, as a countable supremum of continuous functions. Thus, by Lipschitz continuity of $\mu \mapsto \Pi_{\mu}$ with respect to the total variation distance (see Lemma \ref{lem:explicitpimu})
we get that
$$\|\mu_t - \nu_t\|_{\textrm{TV}} \leq \|\mu_0 - \nu_0\|_{\textrm{TV}} + L \int_0^t \|\mu_s - \nu_s\|_{\textrm{TV}} ds$$ for some $L > 0.$ Hence, by the measurable version of Gronwall's inequality (\cite[Theorem 5.1 of the Appendix]{EthierKurtz})
$$\|\mu_t - \nu_t\|_{\textrm{TV}} \leq e^{Lt} \|\mu_0 - \nu_0\|_{\textrm{TV}}$$ and hence there is at most one weak solution with initial condition $\mu_0.$
This, combined with (ii) above shows that $t \rightarrow \Phi_t(\mu)$ is the unique weak solution to (\ref{eq:ODE}).
The semi-flow property $\Phi_{t+s} = \Phi_t \circ \Phi_s$ follows directly from this uniqueness.
\end{proof}
\subsection{Attractors and attractor free sets}
A set $\cal K \subset {\cal P}(\MM)$ is called {\em invariant} under $\Phi$ (respectively {\em positively invariant}) if $\Phi_t(\cal K) = \cal K$ (respectively $\Phi_t({\cal K}) \subset {\cal K}),$ for all $t \geq 0.$

 If $\cal K$ is compact and invariant, then by injectivity of $\Phi$ and compactness, each map $\Phi_t$ maps homeomorphically $\cal K$ onto itself.
 In this case we set $$\Phi^{\cal K}_{t} =  \Phi_t|_{\cal K}$$ for $t \geq 0$ and $$\Phi^{\cal K}_{t}=  (\Phi_{-t}|_{\cal K})^{-1}$$ for all $t \leq 0.$
 It is not hard to check that
 $\Phi^{\cal K} : \R \times {\cal K} \mapsto {\cal K}$ is a flow, $i.e.$ a continuous map such that $\Phi^{\cal K}_t \circ \Phi^{\cal K}_s = \Phi^{\cal K}_{t+s}$ for all $t,s \in \R.$

An {\em attractor} for $\Phi$ is a non empty compact invariant set $A$ having a neighborhood $U_A$ (called a fundamental neighborhood) such that
for every neighborhood $V$ of $A$ there exists $t \geq 0$ such that
$$s \geq t \Rightarrow \Phi_s(U_A) \subset V.$$
Equivalently, if $d$ is a distance metrizing ${\cal P}(\MM)$ $$\lim_{t \rightarrow \infty} d(\Phi_t(\mu), A) = 0,$$
{\bf uniformly} in $\mu \in U_A.$

The {\em basin of attraction} of $A$ is the set $\mathsf{Bas}(A)$ consisting  of points $\mu \in {\cal P}(\MM)$ such that {$\lim_{t \rightarrow \infty} d(\Phi_t(\mu), A) = 0$.}

Attractor $A$ is called {\em global} if its basin is the full space ${\cal P}(\MM).$ It is not hard to verify that there is always a (unique)
{\em global attractor} for $\Phi$ given as
$$A = \bigcap_{t \geq 0} \Phi_t({\cal P}(\MM)).$$

{If ${\cal K}$ denotes} a compact invariant set, an attractor for $\Phi^{\cal K}$ is a non empty compact invariant set $A \subset {\cal K}$ having a neighborhood $U_A$  such that
for every neighborhood $V$ of $A$ there exists $t \geq 0$ such that
$$s \geq t \Rightarrow \Phi_s(U_A \cap {\cal K}) \subset V.$$ If furthermore  $A \neq {\cal K},$ $A$ is called a proper attractor.

{${\cal K}$ is called} {\em attractor free} provided ${\cal K}$ is compact invariant and $\Phi^{\cal K}$ has no proper attractors. Attractor free sets coincide with {\em internally chain transitive sets} and characterize the limit sets of {\em asymptotic pseudo trajectories} (see \cite{BH96,B99}). {Recall that the limit set of $(\mu_n)$ is defined by
$$
L = \bigcap_{n \geq 0} \overline{\left\{\mu_k \ | \ k\geq n \right\}}.
$$
}
 In the present context, by Theorem \ref{prop:asymptraj} of Section \ref{sect:APT}, this  implies  that
 \begin{theo}[{Characterisation of $L$}]
 \label{th:limset}
 Under Hypotheses \ref{hypoK} and \ref{hypogain}, the limit set of $\{\mu_n\}$ is almost surely attractor free for $\Phi.$
 \end{theo}
This theorem,  combined with  elementary properties of attractor free sets, gives the following (more tractable) result.
\begin{coro}[{Limit set and attractors}]
\label{cor:limset} Assume Hypotheses \ref{hypoK} and \ref{hypogain}. Let $L$ be the limit set of $\{\mu_n\}.$ With probability one,
 \begin{itemize}
 \item[(i)]  $L$ is a compact connected invariant set.
 \item[(ii)] If $A$ is an attractor and $L \cap \mathsf{Bas}(A) \neq \emptyset,$ then $L \subset A.$  In particular,  $L$ is contained in the global attractor of $\Phi.$
 \end{itemize}
 \end{coro}

{Note that in the two previous theorems, we do not assume Hypothesis \ref{hypoH3}. In particular, the previous result {may be true in some settings with several QSDs}. This flexibility is, for instance, used in the proof of Theorem \ref{th:2points}.}

\subsection{Global Asymptotic Stability}
The flow $\Phi$ is called {\em globally asymptotically stable} if its global attractor reduces to a singleton $\{\mu^\star\}.$ Observe that, in such a case,  $\mu^\star$ is necessarily the unique equilibrium of $\Phi,$ hence the unique QSD of $K.$

We shall give here sufficient conditions {ensuring} global asymptotic stability.
The main idea is to relate the (nonlinear) dynamics of $\Phi$ to the (linear) Fokker-Planck equation of a nonhomogeneous Markov process on $\MM.$ This idea  is due to Champagnat and Villemonais in \cite{CV14}  where it  was successfully used to prove the exponential {convergence of the conditioned laws and the exponential} ergodicity of the \textit{$Q$-process} for a general almost surely absorbed Markov process.

For all $t \geq 0$ and  $s \in \R$ let $R_{t,s}$ be the bounded operator defined on ${\cal C}(\MM)$ by
$$R_{t,s} f = \frac{e^{(t-s) \Aa}(f g_s)}{g_t}{=\frac{e^{(t-s) \Aa}(f e^{s \Aa} \1)}{e^{t \Aa} \1}}$$
{where $g$ is defined by \eqref{eq:gt}}.
It is easily checked\footnote{One can also note that $R_{t,s}$ is the resolvent of the
 linear differential equation on ${\cal C}(\MM) ~ \dot{u} = \frac{1}{g_t}(\Aa (u g_t) - (\Aa u) g_t)$. {{This explains the unusual order for the indices of $R$ ($w.r.t.$ the standard notation of in-homogeneous Markov processes).}}} that $R_{t,t} = Id$ and $R_{t,s} \circ R_{s, u} = R_{t,u}$ for all $t, s \geq 0$ and $u \in \R.$
Furthermore, for all $t \geq s \geq 0$ $R_{t,s}$ is a Markov operator. That is $R_{t,s} \1 = \1$ and $R_{t,s} f \geq 0$ whenever $f \geq 0.$\smallskip

\noindent To shorten notation we set $$R_t = R_{t,0}.$$ The flow $\widetilde{\Phi}$ and the family $\{R_t\}_{t \geq 0}$ are
 linked by the relation
$$\widetilde{\Phi}_t(\delta_x) = \delta_x R_t$$ for all $t \geq 0$ and $x \in \MM.$
However, note that for an arbitrary $\mu \in {\cal P}(\MM)$ $\widetilde{\Phi}_t(\mu)$ and $\mu R_t$ are not  equal. {Indeed, recall that $\mu R_t f = \int_{\MM} R_tf(x) \mu(dx)$.}
\begin{lem}
\label{lem:suffisGAS}
Let $d_\omega$ be any distance on ${\cal P}(\MM)$ metrizing the weak* convergence.  Assume that
$$\Delta_t := \sup_{x,y \in \MM} d_\omega(\delta_x R_t, \delta_y R_t) \rightarrow 0$$ as $t \rightarrow \infty.$
Then  $\Phi$ is globally asymptotically stable.
\end{lem}
\begin{proof}
{By compactness of $\MM$ the condition $\Delta_t \rightarrow 0$ is independent of the choice of $d_\omega.$} We can then assume that $d_{\textrm{FM}}$ is the {\em Fortet-Mourier distance} (see e.g~\cite{dudley,RKSF}) given as
\beq
\label{eq:defFM}
d_\omega(\mu,\nu) =  \sup \{ |\mu f - \nu f | \: : \|f\|_{\infty} + \textrm{Lip}(f) \leq 1 \}.
\eeq
where  $\textrm{Lip}(f)$ stands for  $\sup_{x \neq y}  \frac{|f(x) -f(y)|}{d(x,y)}.$

Since
$$|\mu R_t f - \nu R_t f| = | \int  (R_t(x) - R_t f (y)) d\mu(x) d \nu(y) | \leq  \sup_{x,y \in \MM} |R_t f(x) - R_t f (y)|,$$
it follows from (\ref{eq:defFM})  that
\beq
\label{eq:supsurmunu}
\Delta_t = \sup_{\mu, \nu \in {\cal P}(M)} d_\omega(\mu R_t, \nu R_t).
\eeq
Fix $\nu \in {\cal P}(\MM).$ Then $$\sup_{s \geq 0} d_\omega(\nu R_{t+s}, \nu R_t) ) = \sup_{s \geq 0} d_\omega( (\nu R_{t+s,t}) R_t, \nu R_t) \leq \Delta_t.$$  This shows that $\{\nu R_t\}_{t   \geq 0}$ is a Cauchy sequence in ${\cal P}(\MM).$ Then $\nu R_t \rightarrow \mu^\star$ for some $\mu^\star$ and for all $\mu \in {\cal P}(\MM)$
$$d_\omega(\mu R_t, \mu^\star) \leq \Delta_t.$$
Now, for all $f \in {\cal C}(\MM)$ {$$|\widetilde{\Phi}_t(\mu) f - \mu^\star f|   =  \left|\frac{\mu \left( (R_t f - \mu^\star f) g_t \right)}{\mu g_t}\right|
 \leq \|R_t f - \mu^\star f\|_{\infty} =
\sup_x |\delta_x R_t f - \mu^\star f|.$$} Therefore $d_\omega(\widetilde{\Phi}_t(\mu), \mu^\star) \leq \Delta_t$
and
$$d_\omega(\Phi_t(\mu), \mu^\star)  \leq \Delta_{\tau_{\mu}(t)} \leq \sup\Big\{ \Delta_s \: : s \geq \frac{t}{\|\Aa\|_{\infty}}\Big \}$$ where the last inequality  follows from the fact that $\dot{s}_{\mu}(t) \leq \|\Aa\|_{\infty}.$ This proves that $\{\mu^\star\}$ is a global attractor for $\Phi.$
\end{proof}
\noindent Recall that $g_t(x) = e^{t \Aa} \1$ (see equation (\ref{eq:gt})).
\begin{lem}
\label{lem:GAS2}
Assume  $\mathbf{H_1}, \mathbf{H_2}, \mathbf{H_3}.$  Assume furthermore that
$$\sum_n \frac{\Psi(g_n)}{\|g_n\|_{\infty}} = \infty$$ where $\Psi$ is the probability measure given by (\ref{eq:smallset}).
Then $\Phi$ is globally asymptotically stable.
\end{lem}
\begin{proof}
We first assume that $U = \MM$ in condition $\mathbf{H_3}.$ That is $\Aa(x,dy) \geq \epsilon \Psi(dy)$ for all $x \in \MM.$
Then, for all $f \geq 0$ and $n \in \N$
$$R_{n+1,n} f = \frac{e^{\Aa} (f g_n)}{e^{\Aa} g_n} \geq \frac{\Aa (f g_n)}{e^{\|\Aa\|_{\infty}}\| g_n\|_{\infty}} \geq  \epsilon
\frac{\Psi(f g_n)}{e^{\|\Aa\|_{\infty}}\| g_n\|_{\infty}}.$$
Let $\Psi_n \in {\cal P}(\MM)$ be defined as
$\Psi_n(f) = \frac{\Psi(f g_n)}{\Psi(g_n)}.$ {W}e get
$$R_{n+1,n}(x,\cdot) \geq \epsilon_n \Psi_n( \cdot )$$ with {$\epsilon_n = \epsilon e^{- \Vert \Aa \Vert_\infty}\frac{\Psi(g_n)}{\|g_n\|_{\infty}}.$}  Thus, reasoning exactly like in the proof of Lemma \ref{lem3}, for all $\mu,\nu \in {\cal P}(\MM)$
$$\parallel\mu R_{n+1,n} - \nu R_{n+1,n}\parallel_{\textrm{TV}} \leq (1-  \epsilon_n)\parallel\mu - \nu\parallel_{\textrm{TV}}$$
and, consequently,
$$\parallel \delta_x R_{n+1} - \delta_y R_{n+1}\parallel_{\textrm{TV}}  \leq 2 \prod_{k = 0}^n  (1-  \epsilon_k).$$
The condition $\sum_n \epsilon_n = \infty$ then implies that $\parallel \delta_x R_{n+1} - \delta_y R_{n+1}\parallel_{\textrm{TV}} \rightarrow 0$ uniformly in $x,y$ as $t\rightarrow\infty.$ In particular, the assumption, hence the conclusion, of Lemma \ref{lem:suffisGAS} is satisfied.

To conclude the proof it remains to show that there is no loss of generality in assuming that $U = \MM$ in $\mathbf{H_3}.$
 By Feller continuity,  and Portmanteau's theorem, for all $n \in \N$ and $\delta > 0$ the set
 $$U(n,\delta) = \{x \in \MM: \: K^n(x,U) > \delta \}$$ is open.
Thus by  $\mathbf{H_3}$ and compactness of $\MM,$ there exist $\delta > 0$ and $n_1, \ldots, n_k \in \N$ such that $$\MM = \bigcup_{i = 1}^k U(n_i,\delta).$$ Let now $x \in \MM.$ Then $x \in U(n_i, \delta)$ for some $i$ and
 $$\Aa(x,dy) \geq \sum_{n \geq 0} K^{n_i + n}(x,dy) =  \int_U K^{n_i}(x,dz) \Aa(z,dy) \geq \epsilon \delta \Psi(dy).$$
\end{proof}
The next proposition shows that under  $\mathbf{H_1}, \mathbf{H_2}, \mathbf{H_3}$ and $\mathbf{H_4}$, the assumptions of the preceding lemma are satisfied.
\begin{prop}[{Convergence of $\Phi$}]
\label{prop:GAS3}
Assume $\mathbf{H_1}, \mathbf{H_2}, \mathbf{H_3}$ and $\mathbf{H_4}.$  Then the assumptions
of Lemma \ref{lem:GAS2} are satisfied. In particular, $\Phi$ is globally asymptotically stable.
\end{prop}
\begin{proof}
By Lemma \ref{lem:expect} \textit{(i)} there exists $N \in \N^*$ and $\Theta < 1$ such that
$$K^{{N}}(x,\MM) \leq \Theta$$ for all $x \in \MM.$
Let $(Z_n)_{n\geq 1}$ be a sequence of i.i.d random variables on $\N$ having a geometric distribution,
$$\PE(Z_n = k) = \Theta^k(1-\Theta), \quad k \geq 0.$$ Let $(U_n)$ be a sequence of i.i.d random variables on $\{0,\ldots,N-1\}$ having a uniform distribution,
$$\PE(U_n = k) = \frac{1}{N}, \quad k = 0, \ldots, N-1,$$ and
let $(N_t)_{t \geq 0}$ be a standard Poisson process with parameter $1.$ We assume that  $(Z_n)_{n\geq 1}, (U_n)_{n\geq 1}, (N_t)_{t \geq 0}$ are mutually independent.

By independence we get  that
\begin{align*}
\E \left[ \frac{K^{\sum_{i = 1}^{N_t} (N Z_i + U_i)}}{\Theta^{\sum_{i = 1}^{N_t} Z_i}} \right] &= \sum_{n\geq 0} \frac{t^n}{n!} e^{-t} \E \left [ \frac{K^{N Z_1 + U_1}}{\Theta^{Z_1}} \right]^n\\
&=\sum_{n\geq 0} \frac{t^n}{n!} e^{-t} \left (\frac{(1-\Theta)}{N} \sum_{k \geq 0} \sum_{r = 0, \ldots, N-1} K^{ N k + r} \right)^n
&= e^{-t} e^{ t\frac{(1-\Theta)}{N} \Aa}.
\end{align*}

To shorten notation, set $s = t \frac{(1-\Theta)}{N}.$ Then,  for all $x \in \MM$
$$\Psi(g_{s}) = \Psi (e^{s \Aa} \1) =  e^t  \E \left[  \frac{ \Psi(K^{\sum_{i = 1}^{N_t} (N Z_i + U_i)} \1)}{\Theta^{\sum_{i = 1}^{N_t} Z_i}} \right]$$
$$\geq e^t \E \left[ C \left (\sum_{i = 1}^{N_t} (N Z_i + U_i)  \frac{ K^{\sum_{i = 1}^{N_t} (N Z_i + U_i)} \1(x)}{\Theta^{\sum_{i = 1}^{N_t} Z_i}} \right ) \right]$$
where the last inequality comes from hypothesis  $\mathbf{H_4}.$

For all $n \in \N, {\bf k} = (k_i) \in \N^{\N^*}$ and ${\bf r} = (r_i) \in \{0, \ldots, N-1\}^{\N^*}$
set
$$F(n, {\bf k}, {\bf r}) =  \frac{ K^{\sum_{i = 1}^{n} (N k_i + r_i)} \1(x)}{\Theta^{\sum_{i = 1}^{n} k_i}} $$
and
$$G(n, {\bf k}, {\bf r}) = C\left (\sum_{i = 1}^{n} (N k_i + r_i) \right)$$
and hence, the preceding inequality can be rewritten as,
$$\Psi(g_{s}) \geq  e^t \E \left[ G(N_t, Z, U) F(N_t, Z, U) \right].$$
Write $(n, {\bf k}, {\bf r}) \leq (n', {\bf k}', {\bf r}')$ when  $n \leq n', k_i \leq k'_i$ and $r_i \leq r'_i.$
The relations $\frac{K^N(x,\MM)}{\Theta} \leq 1$ and $K(x,\MM) \leq 1$  on one hand, and the monotonicity of $C$ on the other hand, imply that
$$F(n, {\bf k}, {\bf r} ) \leq F(n', {\bf k}', {\bf r}')$$
and
$$G(n, {\bf k}, {\bf r} ) \leq G(n', {\bf k}', {\bf r}')$$ whenever $(n, {\bf k}, {\bf r}) \leq (n', {\bf k}', {\bf r}')$
Then, by tensorisation of the classical FKG inequality, and Jensen inequality, we get that
\begin{align*}
\Psi(g_{s}) 	&\geq  e^t \E \left[ G(N_t, Z, U) \right] \E \left[ F(N_t, Z, U) \right]\\
&\geq e^t C\left( t \frac{ N}{1-\Theta} + \frac{N-1}{2}\right) \E \left [ F(N_t, Z, U) \right] = C\left( t \frac{ N}{1-\Theta} + \frac{N-1}{2}\right) g_s(x).
\end{align*}
That is
$$\Psi(g_{s}) \geq C\left(\frac{N^2}{(1-\Theta)^2} s + \frac{N-1}{2}\right) g_s(x),$$
so that the assumptions of Lemma \ref{lem:GAS2} are fulfilled.

\end{proof}

\section{Asymptotic pseudo-trajectory}
\label{sect:APT}
Our aim is now to prove that $(\mu_n)_{n\ge0}$, correctly normalized, is an asymptotic pseudo-trajectory of the flow $\Phi$ defined by (\ref{eq:defPhi}).

\subsection{Background} 
To prove that our procedure {has asymptotically the dynamics} of an ODE, we first need to embed it in a continuous-time process at an appropriate scale. Let us add some notation to explain this point. For $n\geq0$ and $t\geq 0$, set $\tau_n=\sum_{k=1}^n\gamma_k$ and $m(t)=\sup\{k\ge 0,  t\geq \tau_k\}$.  Let  $(\widehat{\mu}_t)_{t\ge0}$, $(\widebar{\mu}_t)_{t\ge0}$, $(\widebar{\epsilon}_t)_{t\ge0}$, $(\widebar{\gamma}_t)_{t\ge0}$ defined for all $n\geq 0$ and $s \in [0,\gamma_{n+1})$ by
$$
\widehat{\mu}_{\tau_n +s} = \left(1 - \frac{s}{\gamma_{n+1}}\right) \mu_{n} +  \frac{s}{\gamma_{n+1}} \mu_{n+1}, \ \widebar{\mu}_{\tau_n +s} =\mu_{n},
$$
$ \widebar{\epsilon}_{\tau_n +s}= \epsilon_n \ \text{ and } \widebar{\gamma}(\tau_n + s) = \gamma_n.$
With this notation, Equation \eqref{eq:algosto} can be written as follows:
$$\widehat{\mu}_t= \mu_0 + \int_0^t h(\widebar{\mu}_s) ds+\int_0^{t} \widebar{\epsilon}_{s} ds$$
with $h(\mu)=-\mu+\Pi_{\mu}$. The aim of this section is now to show that $\widehat{\mu}$ is a pseudo-trajectory of $\Phi$ defined in \eqref{eq:defPhi}. Let $d_\omega$ be a metric on $\mathcal{P}$ whose topology corresponds to the convergence in law {(as for instance the Fortet-Mourier distance defined in \eqref{eq:defFM})}. A continuous map $\zeta : \mathbb{R_+} \to \mathcal{P}$ is called an asymptotic pseudo-trajectory for $\Phi$ if
$$
\forall \, T>0,\quad \lim_{t \to \infty} \left( \sup_{0 \leq s \leq T} d_\omega(\zeta(t+s),\Phi(s,\zeta(t))) \right)=0.
$$
Note that this definition makes an explicit reference to $d_\omega$ but is in fact purely topological  (see \cite[Theorem 3.2]{B99}). In our setting, the asymptotic pseudo-trajectory property can be obtained by the following characterization:

\begin{theo}[Asymptotic pseudo-trajectories]
\label{equiv:prop3.5}
The following assertions are equivalent.
\begin{enumerate}
 \item The function $\widehat{\mu}$ is (almost surely) an asymptotic pseudo-trajectory for $\Phi$.
 \item For all continuous and bounded $f$ and $T>0$,
 \begin{equation}\label{eq:characAPT2}
 \lim_{t \to \infty} \sup_{0 \leq s \leq T } |\int_t^{t+s} \widebar{\epsilon}_{u} f du| =0 \quad\text{a.s.}
 \end{equation}
\end{enumerate}
\end{theo}
\begin{proof}
This is a consequence of \cite[Proposition 3.5]{BLR02}.
\end{proof}
The previous theorem is one of the main differences with the previous article \cite{BC15}. Indeed, in finite state space, the topology of the total variation distance is not stronger than the weak topology.

As in \cite{BC15} and older works on reinforced random walks (see references therein), we now need some properties of solutions of Poisson equations to prove that \eqref{eq:characAPT2} holds. However, in contrast with the finite-space setting of \cite{BC15}, the associated bounds are intricate.

\subsection{Poisson Equation related to ${K_\mu}$}
For a fixed $\mu$ and a given function $f:{\MM}\rightarrow\ER$, let us consider the Poisson equation
\begin{equation}\label{Poissonequa}
 f-\Pi_\mu f= (I-{K_\mu}) g.
 \end{equation}
The existence of a solution $g=Q_\mu f$ to this equation and the smoothness of $\mu \mapsto Q_\mu f$ play an important role for the study of our algorithm.
\tcr{These properties are stated in Lemma \ref{prop:poisequa}. Before, we need to establish the following technical lemma:}
\begin{lem}[Lipschitz property of $\mu \to K^j_\mu$ ]
\label{lem:lip-prel}
For every $\mu,\nu \in \mathcal{P}(\MM)$ and $j\in \mathbb{N}$, we have

\begin{equation}\label{eq:lipkmu}
\sup_{\alpha \in {\cal P}(\MM)} \Vert \alpha K^j_\mu - \alpha K^j_\nu \Vert_{\textrm{TV}} \leq  2^j \Vert \mu - \nu \Vert_{\textrm{TV}},
\end{equation}
and for every bounded function $f$  then

$$
\sup_{x\in \MM} \Vert K^j_\mu(f) - K^j_\nu(f) \Vert_{\infty} \leq  2^j \Vert f \Vert_\infty  \Vert \mu - \nu \Vert_{\textrm{TV}}.
$$
\end{lem}
\begin{proof}
By the definition of the total variation, the second part follows from the first one. We thus only focus on the first statement. For every $j\in \mathbb{N}$, one sets
$$
 \kappa_j(\mu,\nu) = \sup_{\alpha \in P} \Vert \alpha K_\mu^{j} - \alpha K_\nu^{j} \Vert_{\textrm{TV}} .
$$
We have $\kappa_0(\mu,\nu)=0$ and since $K_\mu(.) = K(.)+\delta(.) \mu$ and $\alpha(\delta)\le 1$,
$$
\kappa_1(\mu,\nu) =\sup_{\alpha \in P} \Vert \alpha K_\mu - \alpha K_\nu\Vert_{\textrm{TV}}  = \Vert \alpha(\delta) (\mu- \nu) \Vert_{\textrm{TV}} \leq  \Vert \mu- \nu \Vert_{\textrm{TV}}.
$$
Furthermore, for every $j\ge0$,
\begin{align*}
\Vert \alpha(K_\mu)^{j+1} - \alpha(K_\nu)^{j+1}\Vert_{\textrm{TV}}
&= \Vert \alpha(K+\delta \mu) (K_\mu)^j - \alpha(K+\delta \nu) (K_\nu)^j\Vert_{\textrm{TV}}\\
&= \Vert \alpha K (K_\mu^{j} - K_\nu^{j})   + \alpha(\delta) \mu (K_\mu)^j- \alpha(\delta) \nu (K_\nu)^j\Vert_{\textrm{TV}}\\
&\leq \Vert \alpha K K_\mu^{j} -  \alpha K K_\nu^{j} \Vert_{\textrm{TV}}  + \Vert \alpha(\delta)\mu (K_\mu)^j - \alpha(\delta) \nu (K_\mu)^j\Vert_{\textrm{TV}}\\
&\quad \  + \Vert \alpha(\delta) \nu (K_\mu)^j - \alpha(\delta) \nu (K_\nu)^j\Vert_{\textrm{TV}}\\
&\leq \kappa_j(\mu,\nu) + \Vert \mu - \nu \Vert_{\textrm{TV}} + \kappa_j(\mu,\nu)= 2 \kappa_j(\mu,\nu)+ \Vert \mu - \nu \Vert_{\textrm{TV}}.
\end{align*}
Note that for the last inequality, we again used that $\alpha(\delta)\le 1$ and that for every probabilities $\alpha$, $\beta$ and every transition kernel $P$,
$\|\alpha P-\beta P\|_{\textrm{TV}}\le \| \alpha-\beta\|_{\textrm{TV}}$.
An induction of the previous inequality then leads to:
$$
\forall j\in\mathbb{N},\quad \kappa_j(\mu,\nu) \leq 2^j \Vert \mu - \nu \Vert_{\textrm{TV}}
$$
This yields \eqref{eq:lipkmu}.
\end{proof}

\begin{lem}[Poisson equation]\label{prop:poisequa}
Assume Hypothesis \ref{hypoK}. Let $\mu\in{\cal P}({\MM})$. Let $(P_t^\mu)_{t\ge0}$ be defined by  \eqref{poissonization}. Then, for any measurable function $f:{\MM}\rightarrow\ER$, the Poisson equation \eqref{Poissonequa} admits a solution denoted by $Q_\mu f$ and
defined by
\begin{equation}
\label{eq:def-Q}
 Q_\mu f(x)=\int_0^{+\infty} (P_t^\mu f(x) - \Pi_\mu(f))dt, 
\end{equation}
Furthermore,
\begin{enumerate}[label=(\roman*)]
\item for every $\mu\in {\cal P}(\MM)$, \ $\Vert Q_\mu f \Vert_\infty \leq C \Vert f \Vert_\infty$.
\item for every $\mu,\alpha \in {\cal P}(\MM)$, \ $\vert \alpha Q_\mu f \vert \leq C \Vert f \Vert_\infty \Vert \alpha - \Pi_\mu \Vert_{\textrm{TV}}$.
\item for every $\mu \in{\cal P}(\MM)$ and measurable $f:\MM\rightarrow\ER$, $\Vert Q_\mu f - Q_\nu f \Vert_\infty \leq C_2 \Vert f \Vert_\infty  \Vert \mu- \nu \Vert_{\textrm{TV}}$.
\item for every $\mu,\alpha \in {\cal P}(\MM)$,  $\Vert \alpha Q_\mu  - \alpha Q_\nu  \Vert_{\textrm{TV}} \leq C_2 \Vert \mu- \nu \Vert_{\textrm{TV}}$.
\end{enumerate}
\end{lem}

{Note that our work is closely related to \cite{BLR02} which also investigates the pseudo-trajectory property of a measure-valued sequence. {Nevertheless, the scheme of the proof for the smoothness of the Poisson solutions is significantly different. Indeed, in contrast with \cite[Lemma 5.1]{BLR02}, which is proved using classical functional results (such as the Bakry-Emery criterion), the above lemma (especially $(iii)$ and $(iv)$) is obtained using a refinement of the ergodicity result provided by Lemma \ref{lem3}.}}



\begin{proof}
First, by Lemma \ref{lem3}, the integral in \eqref{eq:def-Q} is well defined. Then, coming back to the definition of $(P_t^\mu)_{t\ge0}$ (see \eqref{poissonization}),
one can readily check that  $(P_t^\mu)_{t\ge0}$ has infinitesimal generator ${\cal L}_\mu$ defined on continuous functions $f:\MM\rightarrow\ER$ by
${\cal L}_\mu f= ({K_\mu}-I) f$.  Without loss of generality, one can assume that $\Pi_\mu(f)=0$. Then, by the Dynkin formula and the commutation and linearity properties, it follows that
$$\forall x\in \MM,\;\forall t\ge0,\quad P_t^\mu f(x)= f(x)+ {\cal L}_\mu \int_0^t  P_s^\mu f(x) ds.$$
Letting  $t$ go to $\infty$ and using again Lemma \ref{lem3} (to ensure the convergence of the right and left hand sides), we deduce that it is a solution to the Poisson equation.\smallskip

Statements  $(i)$ and $(ii)$ are also straightforward consequences of Lemma \ref{lem3}. Thus, in the sequel of  the proof, we only focus on the "Lipschitz" properties $(iii)$ and $(iv)$.\smallskip

\noindent {Without loss of generality, we assume in the sequel that $\|f\|_\infty\le 1$}. By \eqref{eq:K=psi+R}, for every $t\geq 0$
\begin{equation}\label{eq:rmu}
(\alpha - \beta) P_t^\mu = (1-\varepsilon)^{\lfloor t\rfloor} (\alpha P_{t-\lfloor t\rfloor}^\mu - \beta P_{t-\lfloor t\rfloor}^\mu )
 R_\mu^{\lfloor t\rfloor}
\end{equation}
where, with the notation of Lemma \ref{lem3}, $R_\mu$ is given by
\begin{align}
\label{eq:Rmu}
R_\mu
&=\frac{1}{(1-\varepsilon)} \left( e^{-1} \sum_{j\geq 0} \frac{1}{j!} K^{j}_\mu - \varepsilon \mu \right)
\end{align}
The kernels $K_\mu^j$ are Lipschitz continuous with respect to the total variation norm, uniformly in $\alpha \in \mathcal{P}(\MM)$, as it can be checked in Lemma \ref{lem:lip-prel} \tcr{above}. Set
$$
\Xi_n(\mu,\nu) = \sup_{\alpha\in{\cal P}(\MM)} \Vert \alpha R_\mu^{n} -  \alpha R_\nu^n  \Vert_{\textrm{TV}}.
$$
From \eqref{eq:Rmu} and \eqref{eq:lipkmu}, we have

$$
\Xi_1(\mu,\nu) \leq \frac{e}{(1-\varepsilon)} \Vert \mu - \nu \Vert_{\textrm{TV}}.
$$
Now,
\begin{align*}
\Xi_{n+1}(\mu,\nu)
&\leq \sup_{\alpha\in{\cal P}(\MM)} \Vert (\alpha R_\mu^{n}) R_\mu  -  (\alpha R_\mu^n) R_\nu  \Vert_{\textrm{TV}}  +\sup_{\alpha\in{\cal P}(\MM)} \Vert (\alpha R_\mu^{n}) R_\nu - (\alpha  R_\nu^n) R_\nu  \Vert_{\textrm{TV}}\\
&\leq \Xi_{1}(\mu,\nu) + \Xi_{n}(\mu,\nu),
\end{align*}
where for the second term, we used that for some laws $\alpha$ and $\beta$ and for a transition kernel $P$, $\|\alpha P-\beta P\|_{\textrm{TV}}\le \|\alpha-\beta\|_{\textrm{TV}}$.
By induction, it follows that
$$
\Xi_n(\mu,\nu)  \leq \frac{n e}{(1-\varepsilon)} \Vert \mu - \nu \Vert_{\textrm{TV}}.
$$
As a consequence, there exists a constant $C$ such that
$$
\Vert R_\mu^{n}f -  R_\nu^{n}f \Vert_{\infty}\leq C n \Vert \mu- \nu \Vert_{\textrm{TV}}.
$$
and for every $\alpha\in{\cal P}(\MM)$,
\begin{equation}\label{eq:contRmu2}
| \alpha (R_\mu)^{n}f -  \alpha (R_\nu)^n f | \leq C n \Vert \mu- \nu \Vert_{\textrm{TV}}.
\end{equation}
Let us now prove that $\mu \to Q_\mu f(x)$ is Lipschitz continuous. From the definition of $Q_\mu$ and from \eqref{eq:rmu}, we have
\begin{align*}
 Q_\mu f(x) - Q_\nu f(x)
&=\sum_{n=0}^{+\infty} (1-\varepsilon)^{n}{\int_{0}^{1}} \left( (\delta_x - \Pi_\mu) P_r^\mu R^n_\mu  f
  - (\delta_x - \Pi_\nu) P_r^\nu R^n_\nu  f  \right) dr.
\end{align*}
Now, for every $n\ge0$ and $r\in[0,1)$,
{
\begin{align*}
\Big| (\delta_x - \Pi_\mu) P_r^\mu R^n_\mu f &- (\delta_x - \Pi_\nu) P_r^\nu R^n_\nu f  \Big|\\
&\le    |\delta_x (P_r^\mu-P_r^\nu) R^n_\mu  f |+|(\Pi_\mu P_r^\mu- \Pi_\nu P_r^\nu) R^n_\nu f |\\
&+|\delta_x P_r^\nu (R^n_\mu  f-R^n_\nu  f)|+|\Pi_\nu  P_r^\nu (R^n_\mu  f-R^n_\nu  f)|
\end{align*}
The two last terms can be controlled by \eqref{eq:contRmu2} with $\alpha= \delta_x P_r^\nu$ and $\alpha=\Pi_\nu P_r^\nu = \Pi_\nu$ respectively.
For the second one, one can deduce a bound from   Proposition \ref{lem:explicitpimu} $(iii)$ and the fact that $\displaystyle{\sup_{\| g \|_\infty \leq 1} \| R^n_\mu g \|_\infty \leq 1}$. Finally for the first one, using Lemma \ref{lem:lip-prel}  and \eqref{poissonization}, we have
\begin{align*}
| \delta_x (P_r^\mu-P_r^\nu) R^n_\mu  f |
&\leq e^{-r} \sum_{j\geq 0} \frac{r^j}{j!} \Vert \delta_x K_\mu^j - \delta_x K_\nu^j \Vert_{\textrm{TV}}
\leq e^r \Vert \mu- \nu \Vert_{\textrm{TV}}
\end{align*}
.  }
 One deduces that, for some constants $C_1,C_2>0$,
 $$ |Q_\mu f(x) - Q_\nu f(x)|\le  \sum_{n=0}^{+\infty} (1- \varepsilon)^n \ [2 C n\Vert \mu- \nu \Vert_{\textrm{TV}}  + C_1 \Vert \mu- \nu \Vert_{\textrm{TV}}]\leq C_2 \Vert \mu- \nu \Vert_{\textrm{TV}}.$$
Since the Lipschitz constant $C_2$ does not depend on $x$, the statements  $(iii)$ and $(iv)$ easily follow.


\end{proof}

\subsection{Asymptotic Pseudo-trajectories}

\begin{theo}\label{prop:asymptraj} Under Hypotheses \ref{hypogain}  and \ref{hypoK}  $(\widehat{\mu}_t)_{t\ge0}$ is an asymptotic pseudo-trajectory of  $\Phi$ as defined by (\ref{eq:defPhi}).
\end{theo}
\begin{Rq} {\rm
Since the  limit of precompact asymptotic pseudo trajectories is internally chain transitive (see \cite{BH96}, \cite{B99}), this theorem implies Theorem \ref{th:limset}, hence Corollary \ref{cor:limset}.}
  \end{Rq}
\begin{proof}

\noindent By Theorem \ref{equiv:prop3.5}, it is enough to show that for any bounded continuous function $f$, for any $T>0$,
$$\limsup_{t\rightarrow+\infty} \sup_{s\in[0,T]}\left|\int_t^{t+s} \widebar{\epsilon}_{s}(f) ds\right|=0$$
which in turn is equivalent to, for every $m\geq 0$
 $$\limsup_{n\rightarrow+\infty} \max_{j \leq m}\left|\sum_{k=n}^{n+j}\varepsilon_k(f)\right| =0.$$
Here,
$$\varepsilon_{n}(f)= \tcrg{\gamma_{n+1}}\left( f(X_{n+1})-\Pi_{\mu_n} (f)\right)=\tcrg{\gamma_{n+1}}\left(Q_{\mu_n} f(X_{n+1})-K_{\mu_n} Q_{\mu_n} f(X_{n+1})\right).$$
We decompose this term as follows:
\begin{equation}
\varepsilon_{n}(f)=\gamma_{n+1}\Delta M_{n+1}(f)+\Delta R_{n+1}(f)+\gamma_{n+1}\Delta D_{n+1}(f)
\end{equation}
with
\begin{align*}
&\Delta M_{n+1}(f)= Q_{\mu_n} f(X_{n+1})-K_{\mu_n}Q_{\mu_n} f(X_{n})\\
&\Delta R_{n+1}(f)=(\gamma_{n+1}-\gamma_n)K_{\mu_n}Q_{\mu_n} f(X_{n})+ \left(\gamma_n K_{\mu_n}Q_{\mu_n} f(X_{n})-\gamma_{n+1} K_{\mu_{n+1}}Q_{\mu_{n+1}} f(X_{n+1})\right)\\
&\Delta D_{n+1}(f)= \left(K_{\mu_{n+1}}Q_{\mu_{n+1}} f(X_{n+1})-K_{\mu_n} Q_{\mu_n} f(X_{n+1})\right).
\end{align*}
First, let us focus on $\Delta R_{n+1}$.  Using for the first part that $(\gamma_n)_{n\ge0}$ is decreasing and that $(x,\mu)\mapsto K_\mu Q_\mu f(x)$ is (uniformly) bounded (Lemma \ref{prop:poisequa} \textit{(i)}), and a telescoping argument for the second part yields for any positive integer $m$:
 $$\left|\sum_{k=n}^{m} \Delta R_{k}(f)\right|\le C\gamma_n.$$
 Second $(\Delta M_k)$ is a sequence of $({\cal F}_n)$-martingale increments. 
From Lemma \ref{prop:poisequa}, $\Delta M_n(f)$ is bounded (and thus subgaussian). As a consequence, using that $\lim_{n\rightarrow+\infty}\gamma_n\log(n)=0$, one can adapt the arguments of \cite[Proposition 4.4]{B99} (based on exponential martingales) to obtain that
 $$\limsup_{n\rightarrow+\infty} \max_{j \leq m}\left|\sum_{k=n}^{n+j}\gamma_k \Delta M_k(f)\right| =0.$$
Finally, for the last term, one uses that $\mu\mapsto K_\mu$ and $\mu\mapsto Q_\mu$  are Lipschitz continuous.
More precisely, using Lemma \ref{prop:poisequa} \textit{(i)}, \textit{(iii)} and Lemma \ref{lem:lip-prel}, we see that there exists $C>0$ such that
\begin{align*}
&\left|K_{\mu_{n+1}} Q_{\mu_{n+1}} f(X_{n+1})-K_{\mu_n} Q_{\mu_n} f(X_{n+1}) \right| \\
\leq & \left|K_{\mu_{n+1}}Q_{\mu_{n+1}} f(X_{n+1})-K_{\mu_{n+1}} Q_{\mu_n} f(X_{n+1}) \right| \\
  &\quad \ + \left|K_{\mu_{n+1}}Q_{\mu_{n}} f(X_{n+1})-K_{\mu_n} Q_{\mu_n} f(X_{n+1})\right| \\
  \leq & \Vert Q_{\mu_{n+1}} f- Q_{\mu_n} f \|_\infty \\
  &\quad \ + \left\|(K_{\mu_{n+1}}-K_{\mu_n}) (Q_{\mu_n} f)\right\|_\infty \\
\leq &C \Vert f \Vert_\infty \Vert \mu_{n+1} - \mu_n \Vert_{\textrm{TV}}\\
\leq &C \Vert f \Vert_\infty \gamma_{n+1},
\end{align*}
where for the last line, we simply used \eqref{eq:mupasapas}. This ends the proof.
\end{proof}

\section{Proof of the main results}
\label{sec:PMR}

\subsection{Proof of Theorem \ref{theo1}}\label{subsecseptun}

By Proposition  \ref{prop:GAS3}, $\{\mu^\star\}$ is a global attractor for $\Phi$. The result then follows from Corollary \ref{cor:limset}.

\subsection{Proof of Theorem \ref{th:cvmarche}}\label{subsecseptdeux}

Let us assume for the moment that there exist $C>0$ and $\rho \in (0,1)$ such that
for any starting distribution $\alpha$,
\begin{equation}
\label{eq:erggeom}
\forall n\geq 0, \ \Vert \alpha K^n_{\mu^\star} - \mu^\star \Vert_{\textrm{TV}} \leq C \rho^n.
\end{equation}
With this assumption, $\mu^\star$ is a global attractor for the discrete time dynamical system on $\mathcal{P}$ 
 induced by the map $\mu \mapsto \mu K_{\mu^\star}$. Let $\nu_n$ be the law of $X_n$, for $n\geq0$; namely $\nu_n(A)=\mathbb{P}(X_n \in A)$,
for every Borel set $A$. To prove that $\nu_n \to \mu^\star$, it is then enough to prove that the sequence $(\nu_n )_{n\geq 0}$ is an asymptotic pseudo-trajectory of this dynamics; namely that $d_\omega ( \nu_n K_{\mu^\star}, \nu_{n+1}) \to 0$. Indeed, the limit set of a bounded asymptotic pseudo-trajectory is contained in every global attractor (see e.g \cite[Theorem 6.9]{B99} or \cite[Theorem 6.10]{B99}.) So, let us firstly show that for every continuous and bounded function $f$,
\begin{equation}
\label{eq:APTdiscret}
\lim_{n \to \infty} (\nu_{n+1} (f) - \nu_n K_{\mu^\star} f)=0.
\end{equation}
By definition of the algorithm, for every $n\geq0$,
\begin{align*}
\mathbb{E}\left[f(X_{n+1}) \ | \ \mathcal{F}_n\right]
&= K_{\mu_n} f(X_n)
= Kf(X_n) + \mu_n(f) \delta(X_n)
\end{align*}
Taking the expectation, we find
$$
\nu_{n+1}(f) = \nu_n K(f) + \mathbb{E}[\mu_n(f) \delta(X_n)] = \nu_n K_{\mu^\star}(f) + \mathbb{E}[(\mu_n(f) -\mu^\star)(f)) \delta(X_n)] .
$$
But by Theorem \ref{theo1} and dominated convergence theorem,
$$
\lim_{n \to \infty} \mathbb{E}[(\mu_n -\mu^\star)(f) \delta(X_n)] =0,
$$
and hence \eqref{eq:APTdiscret} holds.

We are now free to choose any metric on $\mathcal{P}$ embedded with the weak topology. Let  $(f_k)_{k\geq 0}$ be a sequence of $C^\infty$ functions dense in the space of continuous and bounded (by $1$) functions (with respect to the uniform convergence). Consider the distance $d_\omega$ defined by
$$
d_\omega(\mu,\nu)= \sum_{k \geq 0} \frac{1}{2^k} |\mu(f_k) - \nu(f_k) |.
$$
It is well known that $d_\omega$ is a metric on $\mathcal{P}$ which induces the convergence in law. From \eqref{eq:APTdiscret} and dominated convergence Theorem, we have that $d_\omega ( \nu_n K_{\mu^\star}, \nu_{n+1}) \to 0$.

{It remains to prove inequality \eqref{eq:erggeom}. The proof is similar to Lemma \ref{lem3}. Indeed, by Lemma \ref{lem:expect}, there exists $N\geq 0$ and $\delta_0$ such that for all $x\in \MM$ such that
$\widebar{K}^N (x,\cim)\geq \delta_0 $. Using that $\mu^\star K_{\mu^\star}=\theta_\star\mu^\star$, we have
$$
K^N_{\mu^\star} =  K^N + \widebar{K}^N(\cdot, \cim) \mu^\star
$$
and hence the following lower-bound holds: $\inf_{x\in \MM} K^N_{\mu^\star}(x,\cdot) \geq \delta_0 \mu^\star(\cdot)$. It then implies bound \eqref{eq:erggeom} with the same argument as that of Lemma \ref{lem3}.}

\begin{Rq}[{Periodicity}]
Note that the previous argument shows in particular that  the uniform ergodicity of $K_{\mu^\star}$ is preserved in a non-aperiodic setting. This is the reason why the convergence in distribution of $(X_n)_{n\ge1}$ also holds in this case.
\end{Rq}
\subsection{Proof of Theorem \ref{th:2points}}
The proof relies on the following lemma.
\begin{lem}
\label{lem:ode2points}
Suppose $\Theta_1 > \Theta_2.$ Then
\begin{itemize}
\item[(i)] ${\cal P}(\MM_2)$ is positively invariant under $\Phi$ and $\Phi|{\cal P}(\MM_2)$ is globally asymptotically stable with attractor $\{\mu^\star_2\}.$
    \item[(ii)] There exists another equilibrium for $\Phi$ (i.e another QSD for $K$) $\mu^\star$ having full support (i.e~$\mu^\star(x) > 0$ for all $x \in \MM$). Furthermore $\{\mu^\star\}$ is an attractor whose basin of attraction is ${\cal P}(\MM) \setminus {\cal P}(\MM_2)$.
        \end{itemize}
\end{lem}
\begin{proof}
\textit{(i)}
 It easily follows from  the assumption $\MM_2 \not \hookrightarrow \MM_1,$ and from the definitions of $K_{\mu}$ and $\Pi_{\mu}$  that ${\cal P}(\MM_2)$ is positively invariant under $\Phi.$
By irreducibility of $K_2,$ Lemma \ref{lem:hypo4fini} and Proposition  \ref{prop:GAS3},  $\mu^\star_2$ is then a global attractor for $\Phi|{\cal P}(\MM_2).$

\textit{(ii)} Let $d_i$ be the cardinal of $\MM_i$ and  $d = d_1 + d_2.$ Identifying ${\cal B}(\MM_i, \R)$  (respectively ${\cal B}(\MM, \R)$ ) with column vectors of $\R^{d_i}$ (respectively $\R^{d}$) and ${\cal M}(\MM_i)$ (respectively  ${\cal M}(\MM)$) with row vectors of  $\R^{d_i}$ (respectively $\R^{d}$),
 $K$ can be written  as a $d \times d$ block triangular matrix
$$K = \left(
    \begin{array}{cc}
      K_1 & K_{12} \\
      0 & K_2 \\
    \end{array}
  \right),$$
  where for each $i = 1, 2, K_i$  is a $d_i \times d_i$ irreducible matrix.

  Let $E^l_{\Theta_1}$ and $E^r_{\Theta_1}$ be the left and right eigenspaces associated to $\Theta_1.$ That is  $$E^l_{\Theta_1} = \{ \mu \in {\cal M}(\MM): \: \mu K = \Theta_1 \mu\}$$ and $$E^r_{\Theta_1} = \{ f \in {\cal B}(\MM,\R) : \:  K f = \Theta_1 f\}.$$ We claim that
  \begin{equation}
  \label{eq:El}
  E^l_{\Theta_1} = \R \mu^\star
   \end{equation}
   for some $\mu^\star \in {\cal P}(\MM)$ having full support (i.e~   $\mu(x) > 0$ for all $x$);  and
   \begin{equation}
   \label{eq:Er}
   E^r_{\Theta_1} = \R f^*
    \end{equation} for some $f^* \in {\cal B}(\MM,\R^+)$ satisfying  $$f^*(x) > 0 \Leftrightarrow x \in \MM_1,$$  and $\mu^\star(f^*) = 1.$

Actually,  by irreducibility of $K_1$  and the Perron Frobenius Theorem (for irreducible matrices), $\Theta_1$ is a simple eigenvalue of $K_1$ and there exists $g \in {\cal B}(\MM_1, \R) := \R^{d_1}$ with  positive entries  such that  $K_1 g = \Theta_1 g.$ $\Theta_1$ being strictly larger than the spectral radius $\Theta_2$ of $K_2$, $\Theta_1$ is not an eigenvalue of $K_2$.  Thus, it is also simple for $K$ and (\ref{eq:Er}) holds with $f^*$ defined by $f^*(x) = g(x)$ for $x \in \MM_1$ and $f^*(x) = 0$ for $x \in \MM_2.$

Again by the Perron Frobenius theorem   (but this time for non irreducible matrices) there exists $\mu^\star \in {\cal P}(\MM) \cap E^l_{\Theta_1},$ so that, by simplicity of $\Theta_1,$ (\ref{eq:El}) holds. It remains to check  that $\mu^\star$ has full support. First, observe that $\mu^\star$ cannot be supported by $\MM_2$ for otherwise $\mu^\star$ would be a left eigenvector of $K_2$ and $\Theta_1$ an eigenvalue of $K_2.$  Thus there exists $x \in \MM_1$ such that $\mu^\star(x) > 0,$ but  since $x \hookrightarrow y$, then for all $y{ \in \MM},$ we have $\mu^\star(y) > 0.$

Replacing $f^*$ by $\frac{f^*}{\mu^\star(f^*)}$ we can always assume that $$\mu^\star(f^*) = 1.$$ This ends the proof of the claim.

Let  $(f^*)^{\perp} = \{\nu \in {\cal M}(\MM) : \: \nu(f^*) = 0\}.$
It follows from what precedes that   the splitting
  $${\cal M}(\MM)  = \R \mu^\star \oplus (f^*)^{\perp}$$ is invariant by the map   $\nu \mapsto \nu K,$ hence also by $\nu \mapsto \nu \Aa,$ and $ \nu \mapsto \nu e^{t \Aa}.$
  Let $\Aa^\perp$ denote the operator on  $(f^*)^{\perp}$   defined by $$\nu \Aa^\perp = \nu \Aa - \frac{1}{1-\Theta_1} \nu.$$
  For all  $\mu \in {\cal P}(\MM) \setminus {\cal P}(\MM_2),$ $\mu(f^*) \neq 0$ and  $\mu$ decomposes as $\mu = \mu(f^*) \mu^\star + \widetilde{\mu}$ with $\widetilde{\mu} = \mu - \mu(f^*) \mu^\star \in (f^*)^{\perp}.$ Therefore
  $$\mu e^{t \Aa} = \mu(f^*) e^{\frac{t}{1-\Theta_1} } \mu^\star + \widetilde{\mu} e^{t \Aa} = \mu(f^*) e^{\frac{t}{1-\Theta_1}}  \left ( \mu^\star + \frac{ \widetilde{\mu}}{\mu(f^*)} e^{t \Aa^\perp}\right )$$
  and
  $$\widetilde{\Phi}_t(\mu) := \frac{ \mu e^{t \Aa}}{\mu e^{t \Aa} \1} = \frac{\mu^\star + \frac{ \widetilde{\mu}}{\mu(f^*)} e^{t \Aa^\perp}}{1 + \frac{ \widetilde{\mu}}{\mu(f^*)} e^{t \Aa^\perp} \1}.$$
 Now, remark that any eigenvalue $\Lambda$ of $\Aa^\perp$ writes $$\Lambda =  \frac{1}{1-\lambda} - \frac{1}{1-\Theta_1}$$ where  $\lambda = a + ib$ is an eigenvalue of $K$ distinct from $\Theta_1.$ In particular, $a < \Theta_1 < 1.$ Then,
 $$Re(\Lambda) = \frac{1-a}{(1-a)^2 + b^2} - \frac{1}{1-\Theta_1} \leq \frac{1}{1-a} -  \frac{1}{1-\Theta_1} < 0.$$
The fact that all eigenvalues of  $\Aa^\perp$ have negative real part implies that $\|e^{t \Aa}\| \rightarrow 0$ as $t \rightarrow \infty.$
This proves that $\lim_{t \rightarrow \infty} \widetilde{\Phi}_t(\mu) =  \mu^\star$ and that for every compact set ${\cal K} \subset {\cal P}(\MM) \setminus {\cal P}(\MM_2)$  the convergence is uniform in $\mu \in {\cal K}$ (because $\mu \mapsto \mu(f^*)$ is separated from zero on such a compact).
This shows that $\mu^\star$ is an attractor for $\widetilde{\Phi}$ whose basin is ${\cal P}(\MM) \setminus {\cal P}(\MM_2).$ Proceeding like in the end of the proof of Lemma \ref{lem:suffisGAS} we conclude that the same is true for $\Phi.$
\end{proof}
 We now pass to the proof of   Theorem \ref{th:2points}.\smallskip

\noindent  \textit{(i)} Let $L$ be the limit set of $\{\mu_n\}.$  If $L \subset {\cal P}(\MM_2),$ then  $L = \{\mu^\star_2\}$ because $L$ is compact invariant and, by Lemma \ref{lem:ode2points} (i), $\{\mu^\star_2\}$ is the only compact invariant subset of ${\cal P}(\MM_2)$. If  $L \cap {\cal P}(\MM) \setminus {\cal P}(\MM_2) \neq  \emptyset,$ then by Lemma \ref{lem:ode2points} (ii) and Corollary \ref{cor:limset}, $L = \{\mu^\star\}.$

  \textit{(ii)} If $X_0 \in \MM_2,$ then  $\mu_0 = \delta_{X_0} \in {\cal P}(\MM_2)$ and,  by the definition of $(X_n)$ (see equation (\ref{dynamicsxn})),    $(X_n)$ lives in $\MM_2.$ This implies  that $\mu_n \rightarrow \mu^\star_2$  by assertion   \textit{(i)}.

    \textit{(iii)} If  $X_0 \in \MM_1,$ the point $\mu^\star$ is straightforwardly attainable (in the sense of \cite[Definition 7.1]{B99})  and then, by \cite[Theorem 7.3]{B99}, we have
$$
\mathbb{P}\left(\lim_{n \to \infty} \mu_n = \mu^\star\right)>0.
$$

\textit{(iv)} Let us now prove the last point of Theorem \ref{th:2points} using an \textit{ad hoc} argument under the additional assumption:
\begin{equation}
\label{eq:CNS2pt}
\sum_{n \geq 0} \prod_{i=0}^{n}(1- \gamma_i)  < + \infty.
\end{equation}
If $X_0 \in \MM_2$ there is nothing to prove. We then suppose that $X_0 = x \in \MM_1.$  Clearly there exists $n_0 \geq 1$ such that $X_{n_0} \in \MM_2$ with positive probability.
Using the estimate (\ref{eq:uniformtheta}), the definition of $(X_n)$ and the recursive formula (\ref{eq:mupasapas}) we get that for all $n > n_0$:

$$\PE(X_{n+1} \in \MM_1 | {\cal F}_n) \leq (1- c \Theta_2) \mu_n(\MM_1)$$
almost surely on the event $\{X_n \in \MM_2\}$
and
$$\mu_n(\MM_1) =   \mu_{n_0}(\MM_1)\prod_{i = n_0 + 1}^n (1-\gamma_i) \leq \prod_{i = n_0 + 1}^n (1-\gamma_i)$$ almost surely on the event $$ \{X_{n_0}, X_{n_0 + 1}, \ldots, X_n \in \MM_2\}.$$
Therefore
$$
\mathbb{P}(X_{n_0}, X_{n_0 + 1}, \ldots, X_{n+1} \in \MM_2) = \E \left(\PE(X_{n+1} \in \MM_2 | {\cal F}_n) \1_{ \{X_{n_0}, X_{n_0 + 1}, \ldots, X_n \in \MM_2\}}\right)$$
$$ \geq \left (1 - (1-c\Theta_2) \prod_{i = n_0 + 1}^n (1-\gamma_i) \right)  \PE \left (X_{n_0}, X_{n_0 + 1}, \ldots, X_n \in \MM_2  \right );
$$
and, consequently,
$$\mathbb{P} \left (\forall n \geq n_0 \:  X_n \in \MM_2  \right ) \geq
 \prod_{n \geq n_0 + 1} \left (1 - (1-c\Theta_2)  \prod_{i=n_0 + 1}^{n}(1- \gamma_i)  \right) \PE(X_{n_0} \in \MM_2).$$
The {right hand side of the} previous bound is positive if and only if \eqref{eq:CNS2pt} holds.

\subsection{Proof of Theorem \ref{prop:diff1}}
We begin by recalling a classical lemma about the $L^2$-control of the distance between the Euler scheme and the diffusion (see $e.g.$ \cite[Theorem B.1.4]{bouleau-lepingle} for a very close statement).
\begin{lem}\label{lem:contL2Euler}
Assume that $b$ and $\sigma$ are Lipschitz continuous functions. Then, for every positive $T$, there exists a constant $C(T)$ such that for every starting point $x$ of $\ER^d$,
$$\ES_x\left[\sup_{t\in[0,T]}|\xi_t^h-\XX_t|^2\right] \le C(T)(1+|x|^2)  h.$$
\end{lem}
We continue with some uniform controls of the exit time of $D$. For a given set $A$ and a path $w:\ER_+\mapsto\ER^d$, we denote by $\tau_A(w)$ the exit time of $A$ defined by:
$$\tau_A(w)=\inf\{t>0,w(t)\in A^c\}.$$
\begin{lem}\label{lem:consunifellip}
Assume that $b$ and $\sigma$ are Lispchitz continuous functions and that $\sigma\sigma^*\ge \rho_0 I_d$. Then,\smallskip

\noindent

\noindent (i) Let $\delta>0$ and set $D_\delta=\{x\in\ER^d, d(x,D)\le \delta\}$. For each $t_0>0$, we have
$$\sup_{x\in D} \PE_x(\tau_{D_\delta}(\xi)>t_0)<1.$$

\noindent (ii)  There exist some compact subsets ${\cal K}$ and $\widetilde{{\cal K}}$ of $D$ such that ${\cal K}\subset \widetilde{{\cal K}}$ and such that there exist some positive $t_0$, $t_1$, such that for all $h>0$,

$$\forall x\in D,\; \PE_x\left({\xi}_{t_0}\in {\cal K},  \tau_D(\xi^h)>t_0\right)>0$$ and
$$\inf_{x\in {\cal K}}\PE_x( \xi_{t_1}\in {\cal K}\,\textnormal{and}\,\forall t\in[0,t_1], {\xi}_{t}\in \widetilde{{\cal K}}) >0.$$


\end{lem}
\begin{proof} $(i)$ Using the fact that for every $t_0>0$, $\sup_{s\in[0,t_0]}|\xi_s^x-\xi_s^{x_0}|\rightarrow0$ in probability when $x\rightarrow x_0$, one deduces from the dominated convergence theorem that
$x\mapsto \PE_x(\tau_{D_\delta}(\xi)>t_0)$ is continuous on $\widebar{D}$. As a consequence, it is enough to show that  for every $x\in D$, $\PE_x(\tau_{D_\delta}(\xi)>t_0)<1$. This last point is a consequence of the ellipticity condition.\smallskip

\noindent $(ii)$ Let us begin by the first statement. Let $t_0>0$. Since the Euler scheme is stepwise constant, it is enough to show that $\PE_x(\xi^h_{[t_0/h] h}\in {\cal K}, \xi_{\ell h}^h\in D,\ell \in \{ 0,\ldots,[t_0/h]\})>0$. This follows easily from the fact that, under the ellipticity condition, the transition kernel of the discrete Euler scheme is a (uniformly) non-degenerated Gaussian with bounded bias (on compact sets). \smallskip

\noindent For the second statement, let $x_0\in D$, ${\cal K}=\widebar{B}(x_0,r)$ and $\widetilde{{\cal K}}=\widebar{B}(x_0,2r)$ where $r=\frac 14 d(x_0,\partial D)$. For every $x\in {\cal K}$, let  $\psi^{x,x_0}:[0,1]\rightarrow\ER$ denote the function defined by
$\psi^{x,x_0}(t)=t x_0+(1-t) x$, $t\in[0,1].$ Let $t_1>0$. Since $\psi^{x,x_0}$ is ${\cal C}^1$ and $\sup_{x\in {\cal K},t\in[0,1]}|\partial_t \psi^{x,x_0}|<+\infty$, it is well-known (see $e.g.$ \cite[Theorem 8.5]{bass_book}) that
for all $\varepsilon>0$, there exists a positive $c_\varepsilon$ such that
$$\forall x\in {\cal K},  \quad \PE_x(\sup_{t\in[0,t_1]}|\xi_t-\psi^{x,x_0}(t)|\le \varepsilon)\ge c_\varepsilon.$$
Taking $\varepsilon= \frac{r}{2}$, the result follows.

\end{proof}


\textit{Proof of Theorem \ref{prop:diff1}.} First, let us remark that by the ellipticity condition, we have for every starting point $x$ of $D$, $\tau_D(\xi^h)=\tau_{\widebar{D}}(\xi^h)$ $a.s.$ Actually, as mentioned before, ${\cal L}(Y_{k+1}^h|Y_k^h=x)$  is a non-degenerate Gaussian random variable, which implies that $\PE(Y_{k+1}^h\in \partial D|Y_k^h\in D)=0$. Furthermore, if $x\in\partial D$, $\tau_{\widebar{D}}(\xi^h)=0$.  Then, if one denotes by $\mu_h^\star$ the unique QSD of ${Y}^h$ killed when leaving ${\widebar D}$,  it follows that $\mu_h^\star(\partial D)=0$.
Without loss of generality, one can thus work with $D$ instead of $\widebar{D}$ in the sequel. \\

Let $K^h$ denote the sub-Markovian kernel related to the discrete-time Euler scheme  killed when it leaves $D$. Namely, for every bounded Borelian function $f:\ER^d\rightarrow\ER$,
$$K^h f(x)=\ES[f(x+h b(x)+\sqrt{h}\sigma(x)Z)1_{\{x+h b(x)+\sqrt{h}\sigma(x)Z\in D\}} ]$$
where $Z\sim{\cal N}(0,I_d)$. Let $\rho_h$ denote the extinction rate related to $\mu_h^{\star}$. We have
\begin{equation}\label{eq:caracQSD2}
\int K^h f(x)\mu_h^\star(dx)=\rho_h \mu_h^\star(f).
\end{equation}
Setting $\lambda_h=\log(\rho_h)/h$, it easily follows from an induction that for every positive $t$, for every bounded measurable function $f:D\rightarrow\ER$,
\begin{equation}\label{eq:caracQSDapp}
\ES_{\mu_h^{\star}}[f(\xi_t^h)1_{\tau_D(\xi^h)>t}]=C_h(t)\exp(-\lambda_h t)\mu_h^{\star}(f)
\end{equation}
where $C_h(t)=\exp(\frac{t}{h}-\lfloor \frac{t}{h}\rfloor)$.\smallskip

The aim is now to first show that $(\mu_h^\star)_h$ is tight on the open set $D$ and then, to prove that every weak limit  $\mu$   (for the weak topology induced by the usual topology on $D$) is a QSD. The convergence will follow from the uniqueness of
$\mu^\star$ given in Theorem 5.5 of  \cite[Chapter $3$]{P95}. 
This task is divided in three steps:

\textbf{Step 1} (Bounds for $\lambda_h$):   We prove that there exist some positive $\lambda_{\rm min}$, $\lambda_{\rm max}$ and $h_0$ such that for any $h\in(0,h_0)$, $\lambda_h$ defined in \eqref{eq:caracQSDapp} satisfies
$\lambda_{\rm min}\le \lambda_{h}\le\lambda_{\rm max}$. Let us begin by the lower-bound. For every $x\in D$ and $\delta>0$,
\begin{align*}
\PE_x(\tau_D(\xi^h)> t_0)&\le \PE_x(\{\tau_{D_\delta}(\xi)> t_0 \} \cup \{ \sup_{t\in[0,t_0]}|\xi^h-\xi | \geq \delta\})\\
&\le \PE_x(\tau_{D_\delta}(\xi)> t_0)+ \PE_x\left( \sup_{t\in[0,t_0]} |\xi^h-\xi | \geq \delta\right).
\end{align*}
By Lemma \ref{lem:consunifellip}(i) and Lemma \ref{lem:contL2Euler}, one easily deduces that for $h$ small enough,
$$a:=\sup_{x\in D} \PE_x(\tau_D(\xi^h) > t_0) <1.$$
Recalling that $\xi^h$ is stepwise constant, note that $t_0$ can be replaced by $t_0^h=\lfloor t_0/h\rfloor h$ in the previous inequality. 
Then,
\begin{align*}
\PE_x(\tau_D(\xi^h)>  k t_0^h)&=\PE_x(\tau_D(\xi^h)> k t_0^h|\tau_D(\xi^h)> (k-1)t_0^h)\PE_x(\tau_D(\xi^h)> (k-1)t_0^h)\\
&\le a \PE_x(\tau_D(\xi^h)> (k-1)t_0^h).
\end{align*}
By induction, it follows that there exists $h_0\in(0,t_0)$ such that for every $t>0$ and $h\in(0,h_0)$,
$$\PE_x(\tau_D(\xi^h)>t)\le a^{\lfloor t/t_0^h\rfloor}\le C \exp (-\lambda_{\rm min} t)$$
with $\lambda_{\rm min}=-\log(a)/(t_0-h_0)$. Since the right-hand side does not depend on $x$, one deduces that $\PE_{\mu_h^\star}(\tau_D(\xi^h)>t)\le C \exp (-\lambda_{\rm min} t).$ Since this inequality holds for every $t$ (with $C$ not depending on $t$), it follows from \eqref{eq:caracQSDapp} that $\lambda_{\rm min} <\lambda_h$ (using that $C(t)\le e$). This yields the lower-bound. \smallskip
As concerns the upper-bound, one first deduces from  Lemmas \ref{lem:contL2Euler} and  \ref{lem:consunifellip}(ii)   that there exists a compact subset ${\cal K}$ of $D$ such that there exist some positive $t_0$, $t_1$, $h_0$ and $\varepsilon$, such that for every $x\in D$,  and $h\in(0,h_0)$,
$\PE_x(\xi^h_{t_0}\in {\cal K},\tau_D(\xi^h)>t_0)>0$ and
\begin{equation}\label{eq:cdndion}
\inf_{y\in {\cal K}}\PE_y( \xi^h_{t_1}\in {\cal K}\,\textnormal{and}\,\forall t\in[0,t_1], \xi^h_{t}\in D)\ge\varepsilon.
\end{equation}
Once again, using the fact that $\xi^h$ is stepwise constant, one sets  $t_0^h=\lfloor t_0/h\rfloor h$ and $t_1^h=\lfloor t_1/h\rfloor h$. We have
$$\PE_x(\tau_D({\xi^h})>t)\ge\PE_x(\xi^h_{t_0^h+ \ell t_1^h}\in {\cal K}, \ell\in\{0,\ldots, N_h(t)\}, \tau_D({\xi^h})>t)$$
with $N_h(t)=\inf \{\ell, t_0^h+\ell t_1^h>t\}$. Using  the Markov property and an induction, it follows from \eqref{eq:cdndion} that, for $h$ small enough,  for every $t>t_0^h$, for every $x\in D$,
\begin{align*}
\PE_x(\tau_D({\xi^h})>t)&\ge \PE_x(\xi^h_{t_0}\in {\cal K},\tau_D(\xi^h)>t_0) \varepsilon^{N_h(t)}\\
&\ge C\PE_x(\xi^h_{t_0}\in {\cal K},\tau_D(\xi^h)>t_0)\exp\left(\frac{\log(\varepsilon)}{2 t_1} t\right),
\end{align*}
where in the last inequality, we used that for $h$ small enough
$$N_h(t)=\left\lfloor \frac{t-t_0^h}{t_1^h}\right\rfloor +1\ge \frac{t}{t_1^h} -\frac{t_0^h}{t_1^h}\ge \frac{t}{2t_1}-\delta,\quad \delta>0.$$
Set $\lambda_{\rm max}=-\frac{\log(\varepsilon)}{2t_1}$. Since $\PE_x(\xi^h_{t_0}\in {\cal K},\tau_D(\xi^h)>t_0)>0$ for every $x\in D$, it follows from what precedes that
for every $t>t_0$,
 $$\PE_{\mu_h}(\tau_D({\xi^h})>t)\ge c\exp(-\lambda_{\rm max} t),$$
 where $c$ is a positive constant (which does not depend on $t$). By  \eqref{eq:caracQSDapp}, one can conclude that $\lambda_h<\lambda_{\rm max}$ (using that $C(t)\ge 1$).\smallskip

\textbf{Step 2} (Tightness of $(\mu_h^\star)$):
We show that $(\mu_h^\star)_{h\in(0,h_0]}$ is tight on $D$. For $\delta>0$, set $B_\delta:=\{x\in D, d(x,\partial D)\le \delta\}$. We need to prove that for every $\varepsilon>0$,
there exists  $\delta_\varepsilon>0$ such that for every $h\in(0,h_0)$, $\mu_h^\star(B_{\delta_\varepsilon})\le\varepsilon$. First, by \eqref{eq:caracQSDapp} (applied with $t=1$) and Step $1$, 
$$\mu_h^\star(B_{\delta})\le \ES_{\mu_h^\star}\left[\1_{\{\xi_1^h\in B_\delta\}\cap \{\tau_D(\xi^h)>1\}}\right]\le \PE_{\mu_h^\star}(\xi_1^h\in B_\delta).$$
But, under the ellipticity condition, ${\cal L}(\xi_1^h|\xi^h_0=x)$ admits a density $p_1^h(x,.)$ w.r.t. the Lebesgue measure $\lambda_d$ and by \cite[Theorem 2.1]{lemaire-menozzi} (for instance),
$$\sup_{x,x'} p_1^h(x,x')\le C$$
where $C$ does not depend on $h$. As a consequence, for every $x\in D$,
$$ \PE_{\mu_h^\star}(\xi_1^h\in B_\delta)\le \int p_1^h(x,x')\lambda_d(dx')\le C\lambda_d(B_\delta).$$
The tightness follows.

\textbf{Step 3} (Identification of the limit): Let $(\mu_{h_n}^\star)_n$ denote a convergent subsequence to $\mu$. One wants to show that $\mu=\mu^\star$  (where $\mu^\star$ stands for the unique QSD of the diffusion killed when leaving $D$). To this end, it remains to show that there exists $\lambda>0$ such that for any positive $t$ and any bounded  continuous function $f: D \rightarrow\ER$,
\begin{equation}\label{eq:caracQSD}
\ES_{\mu}[f({\xi}_t)1_{\tau_D(\xi)>t}]=\exp(-\lambda t)\mu(f).
\end{equation}
With standard arguments, one can check that this is enough to prove this statement when $f$ is ${\cal C}^2$ with compact support in $D$.\smallskip

Let us consider Equation \eqref{eq:caracQSDapp}. First, up to a potential extraction, one can deduce from Step $1$ that $\lambda_{h_n}\rightarrow\lambda\in\ER$. By the weak convergence of $(\mu_{h_n}^\star)_n$, it follows that the right-hand side of
 \eqref{eq:caracQSDapp} satisfies:
\begin{equation}\label{eq:rhscar}
 C_h(t)\exp(-\lambda_h t)\mu_h^{\star}(f)\rightarrow\exp(-\lambda t)\mu(f).
 \end{equation}
 \smallskip
Second, by \cite{gobet} (Theorem 2.4 and remarks of Section 6 therein about the hypothesis of this theorem),  there exists a constant $C_1(t)$ such that for all $x\in D$,
   $$|\ES_x[f(\xi_t^h)1_{\tau_D(\xi^h)>t}]-\ES_x[f({\xi}_t)1_{\tau_D(\xi)>t}]|\le C_{f} C_1(t)\sqrt{h}.$$

It follows  that
$$\sup_{x\in D} |\ES_x[f(\xi_t^h)1_{\tau_D(\xi^h)>t}]-\ES_x[f({\xi}_t)1_{\tau_D(\xi)>t}]|\xrightarrow{h\rightarrow0}0.$$
As a consequence,
$$ \ES_{\mu_h^\star}[f(\xi_t^h)1_{\tau_D(\xi^h)>t}]-\ES_{\mu_h^\star}[f({\xi}_t)1_{\tau_D(\xi)>t}]\xrightarrow{h\rightarrow0}0.$$
Now, by a dominated convergence argument (using that $\sup_{s\in[0,t]}|\xi_s^x-\xi_s^{x_0}|\rightarrow0$ in probability when $x\rightarrow x_0$), one remarks that $x\mapsto\ES_x[ f({\xi}_t)1_{\tau_D(\xi)>t}]$ is (bounded) continuous on $D$. As
a consequence,  $\ES_{\mu_{h_n}^\star}[f({\xi}_t)1_{\tau_D(\xi)>t}]\rightarrow\ES_{\mu}[f({\xi}_t)1_{\tau_D(\xi)>t}]$ so that
$$\ES_{\mu_{h_n}^\star}[f({\xi}_t^{h_n})1_{\tau_D(\xi^{h_n})>t}]\xrightarrow{h\rightarrow0}\ES_{\mu}[f({\xi}_t)1_{\tau_D(\xi)>t}].$$
Equality \eqref{eq:caracQSD} follows by plugging  the above convergence and \eqref{eq:rhscar} into \eqref{eq:caracQSDapp}.

\section{Extensions}\label{sec:extensions}
\subsection{Non-compact case: Processes coming down from infinity}
In the main results, we chose to restrain our considerations to compact spaces.  When ${\cal E}$ is only locally compact,   the results of this paper  could be extended
to  the class of processes which \textit{come down from infinity} (CDFI), $i.e.$ which have the ability to come back to a compact set in a bounded time with a uniformly lower-bounded probability (for more details, see $e.g.$ \cite{bansaye-meleard,CCLMM,CMMM11}).  First, note that (CDFI)-condition is a usual and sharp assumption which ensures uniqueness of the QSD in the locally compact setting.  Second, the (CDFI)-condition is in particular ensured if $\MM$ is locally compact and if there exists a map $V$, such that $\{V\leq C\}$ is compact for every $C>0$	 and $M : = \sup_{x\in\MM} KV(x)$ is finite. Also, let us remark that $\bar{\cal P}:=\{\mu\in{\cal P}({\cal E}), \mu(V)\le M\}$ is compact for the weak convergence topology (owing to the coercivity condition on $V$) and is invariant under  the action of the kernel $K$. Then, on this subspace $\bar{\cal P}$,   the main arguments of the proof of the main results could be adapted to obtain the convergence of the algorithm.

\subsection{Non-compact space: the minimal QSD}
In Theorem \ref{th:2points}, we have seen that when a process admits several QSDs, our algorithm may select all of its QSDs with positive probabilities. When (CDFI)-condition fails in the non-compact setting (think for instance about the real Ornstein-Uhlenbeck process killed when leaving $\ER_+$), uniqueness generally fails and one can not  expect the algorithm  to select only one QSD. However, if the aim is to approximate the so-called minimal QSD, namely the one associated to the minimal eigenvalue and appearing in the \textit{Yaglom limit},   then, one can use a compact approximation method in the spirit of  \cite{V11-ejp}. More precisely,  consider for instance a diffusion process $(\xi_t)_{t\ge0}$ on $\ER^d$ killed when leaving an unbounded domain $D$ and denote by $\mu^\star$ the related minimal QSD (when exists). Let also $(K_n)_{n\ge1}$ be an increasing sequence of compact spaces such that $\bigcup_{n\ge1} K_n=D$. Then, under some non-degeneracy assumptions (see $e.g.$ Theorem \ref{prop:diff1}),  the QSD  $\mu_n^\star$ related to $K_n$ is unique for every $n$ and  by \cite[Theorem 3.1]{V11-ejp}, $\lim_{n\rightarrow+\infty}\mu_n^\star=\mu^\star$. Then, using our algorithm for an approximation of $\mu_n^\star$ would lead to an approximation of $\mu^\star$ for $n$ large enough.




\subsection{Continuous-time algorithm}
In view of the approximation of the QSD of a diffusion process $(\xi_t)_{t\geq0}$ on a bounded domain $D$ satisfying the assumptions of Theorem \ref{prop:diff1}, it may be of interest to study the convergence of a continuous-time equivalent of our algorithm (instead of considering an Euler scheme with constant step). Of course, without discretization,  such a problem is mainly theoretical but it is worth noting that the difficulties mentioned below should be very similar if one investigated an algorithm with decreasing step (on this topic, see also Remark \ref{rq:decreasstep}).\smallskip
\noindent The continuous-time algorithm is defined as follows:

\begin{itemize}
\item let $x \in D$ and $(X^1_t)_{t\geq 0}$ be as $(\xi_t)_{t\geq0}$ with initial condition $\xi_0=x$;
\item let $\tau^1 = \inf\{t\geq 0 \ | \ X^1_t \notin D \}$, for all $t< \tau^1$ we set $X_t = X^1_t$;
\item Let $\mu_t = \frac{1}{t} \int_0^t \delta_{X_s} ds$ be the occupation measure of $X$;
\item Let $U$ be a random variable distributed as $\mu_{\tau^1}$ (conditionally \tcr{on the stopping time $\sigma$-field $\mathcal{F}_{\tau^1}$});
\item we set $X_{\tau^1}=U$;
\item the process then evolves  as above starting from $U$.
\end{itemize}
We denote by $(\tau^k)_{k\geq 1}$ the sequence of jumping times. At the $n^{\text{th}}$ time, the process jumps uniformly over the positions from all its past and not only from $[\tau^{n},\tau^{n+1}]$. This sequence of stopping times is increasing and almost surely converges to some $\tau^\infty \in (0,+ \infty]$. The process $(X_t)_{t\geq0}$ is well defined until the time $\tau^\infty$. It is not trivial that $\tau^\infty = \infty$ because the process $(X_t)_{t\geq 0}$ can be arbitrarily close to the boundary and the  times between jumps become arbitrarily short. Nonetheless, we have

\begin{lem}[Non-explosion of the continuous-time algorithm]
\label{lem:non-explosion}
Under the assumptions of Theorem \ref{prop:diff1}, we have $\tau^\infty=+\infty$ a.s.
\end{lem}
The proof is given below. This type of problem is reminiscent of the Fleming-Viot particle system \cite{BBF,BBP, V11}. However, the comparison stops here because our procedure is not "Markovian" and their proofs can not be adapted.


\begin{proof}[Proof of Lemma \ref{lem:non-explosion}]
Let $x\in D$ be the starting point of $(X_t)_{t\geq0}$. Fix $\varepsilon>0$ such that $d(x, \partial D) \tcrg{>} \varepsilon$ and choose $0<\delta<\varepsilon$. Let $(B_t)_{t \geq 0}$ a Brownian motion and choose $z$ in the ball $B(x,\delta)$ of center $x$ and radius $\delta$. We let $(\xi^z_t)_{t\geq 0}$ be the solution of
$$
d\xi^z_t=b(\xi^z_t)dt+\sigma(\xi^z_t) dB_t, \quad \XX_0^z =z,
$$
and
$$
\tau_{\varepsilon,\delta,x} = \inf_{z \in B(x,\delta)} \inf\{t\geq 0 \ | \ \xi^z_t \notin B(x,\varepsilon) \}.
$$

The variable $\tau_{\varepsilon,\delta,x}$ is almost-surely positive. On $\{ \tau^\infty <+ \infty\}$, we have, for every $t\in [\tau^1, \tau^\infty)$,
$$
\mu_t(B(x, \delta)) \geq \frac{\tau_{\varepsilon,\delta,x}}{\tau^\infty}.
$$
As a consequence on $\{ \tau^\infty <+ \infty\}$, the process $(X_t)_{t\geq 0}$ jumps infinitely often in $B(x, \delta)$. But if it starts from a point $z\in B(x, \delta)$, its absorption time can be bounded from below by a random variable $\sigma$ (independent from the past) such that $\sigma$ has the same law as $\tau_{\varepsilon,\delta,x}$. Hence, we have
$$
\tau^\infty \geq \sum_{n\geq 1} \sigma_n,
$$
on $\{\tau^\infty < + \infty\}$, where $(\sigma_n)_{n\geq 1}$ is a sequence of i.i.d. random variable distributed as $\tau_{\varepsilon,\delta,x}$. As they are positive, the strong law of large numbers ensures that $\sum_{n\geq 0} \sigma_n=+\infty$ almost surely and then $\mathbb{P}(\tau^\infty <+ \infty)=0$.
\end{proof}

\noindent\textbf{Acknowledgements.} The first author thanks the SNF for the grants  200020/149871 and  200021/175728. The third author thanks the Centre Henri Lebesgue ANR-11-LABX-0020-01 for its stimulating mathematical research programs.


\bibliographystyle{abbrv}
\bibliography{ref}

\end{document}